\documentclass[reqno,11.5pt]{article}

\usepackage{diagrams}
\usepackage{amssymb}
\usepackage{tikz}
\usepackage{amsmath}
\usepackage{latexsym}
\usepackage{mathrsfs}
\usepackage{amsthm}
\usepackage{verbatim}
\usepackage{graphicx}
\usepackage{epstopdf}
\usepackage{epsfig}
\diagramstyle[labelstyle=\scriptstyle]
\usepackage{color}
\usepackage{subfigure}
\usepackage{float}
\usepackage[algo2e,linesnumbered,ruled]{algorithm2e}
\usepackage{titlesec}
\usepackage{bm}
\usepackage[colorlinks,linkcolor=blue,anchorcolor=green,citecolor=red]{hyperref}
\usepackage{amsfonts}
\usepackage{fancyhdr}
\usepackage{amscd}
\usepackage{cases}
\usepackage{amsmath,stackrel}

\usepackage{mathtools}

\newtheorem{theorem}{Theorem}
\newtheorem{remark}{Remark}

\newtheorem{definition}{Definition}

\newtheorem{lemma}{Lemma}

\newcommand{\0}{\mathaccent23}

\newcommand\grad{\operatorname{grad}}
\renewcommand\div{\operatorname{div}}
\newcommand\curl{\operatorname{curl}}
\newcommand\rot{\operatorname{rot}}

\newcommand\V{{\mathcal{V}}}
\newcommand\E{{\mathcal{E}}}
\newcommand\T{{\mathcal{T}}}

\makeatletter
\@addtoreset{equation}{section}
\makeatother

\usepackage{geometry}
\newcommand{\newcase}[1]{
\bigskip
\noindent -- \emph{#1}
}

\begin{document}

\title{Nodal Finite Element de Rham Complexes}

\author{Snorre H. Christiansen\thanks{Department of Mathematics, University of Oslo, PO Box 1053 Blindern, NO 0316 Oslo, Norway. 
              email:{\tt snorrec@math.uio.no}} \and Kaibo Hu\thanks{Department of Mathematics,
University of Oslo,Oslo 0316, Norway. email: {\tt kaibohu@math.uio.no.}}\and Jun Hu\thanks{LMAM and School of Mathematical Sciences, Peking University, Beijing 100871, P. R. China.
              email:{\tt hujun@math.pku.edu.cn}}}
\date{}

\maketitle

\begin{abstract}
We construct 2D and 3D finite element de Rham sequences of arbitrary polynomial degrees with extra smoothness. 
Some of these elements have nodal degrees of freedom (DoFs) and can be considered as generalisations of scalar Hermite and Lagrange elements. Using the nodal values, the number of global degrees of freedom is reduced compared with the classical N\'{e}d\'{e}lec and Brezzi-Douglas-Marini (BDM) finite elements, and the basis functions are more canonical and easier to construct. Our finite elements for ${H}(\mathrm{div})$ with regularity $r=2$ coincide with the nonstandard elements given by  Stenberg (Numer Math 115(1): 131-139, 2010).
We show how regularity decreases in the finite element complexes, so that they branch into known complexes. The standard de Rham complexes of Whitney forms and their higher order version can be regarded as the family with the lowest regularity. The construction of the new families is motivated by the finite element systems.

\end{abstract}

\section{Introduction}

Differential complexes are an important tool in the study of finite element methods. Finite element differential complexes characterise finite element spaces and the operators among them, specifying their kernels and images, which are crucial for the stability of numerical formulations \cite{Arnold2002} and fast solvers \cite{Hiptmair2007}. There are many existing finite element differential complexes, for example the de Rham complex \cite{Hiptmair.R.2002a,Arnold2006}, the Stokes complex \cite{guzman2013conforming,Neilan2015}, the Darcy-Stokes complex \cite{mardal2002robust,tai2006discrete}, the elasticity complex \cite{Arnold2006a,Arnold2008} etc. 
Among them, the discrete de Rham sequence is probably the most fundamental one.  Stokes complexes  and Darcy-Stokes complexes have the same differential operators as the standard de Rham complexes, and the only difference is that the spaces of Stokes and Darcy-Stokes have higher continuity.

There have been many discussions on finite element de Rham sequences. For incomplete polynomials, there are the N\'{e}d\'{e}lec  elements of the first kind \cite{Nedelec.J.1980a} and the Raviart-Thomas  elements \cite{Raviart.P;Thomas.J.1977a}. For complete polynomials, there are the N\'{e}d\'{e}lec elements of the second kind \cite{Nedelec.J.1986a} and the Brezzi-Douglas-Marini (BDM) elements \cite{brezzi1985two}.
  All these successful elements can be unified as discrete differential forms \cite{hiptmair2001higher,Arnold.D;Falk.R;Winther.R.2006a}.  The construction of degrees of freedom (DoFs) for higher order Whitney forms is based on moments on subsimplexes, therefore commuting interpolations can be constructed easily. A periodic table has been developed to include arbitrary polynomial degree for any $k$-forms, for simplicial and tensor product elements in any dimension \cite{arnold2014periodic}.    

Such elements are usually called ``vector elements'', and cannot be represented by Lagrange nodal basis functions. This leads to complications  in applications, especially for high order methods. 

A canonical nodal basis is attractive from the perspective of implementation.  The choice of basis is not unique and there are several important criterions for a good basis, including condition number, sparsity of stiffness and mass matrices, efficient evaluation and rotational symmetry etc. There is a huge literature on the construction of bases for high order finite elements. We refer to the book \cite{karniadakis2013spectral} for a survey of the scalar case. Bases for  $H(\curl)$ and $H(\div)$ are more complicated. Some efforts in this direction can be found in \cite{zaglmayr2006high}.
The difficulty is that the bubble functions of $H(\curl)$ and $H(\div)$ elements are not as  canonical as those of the scalar $H^{1}$ elements  \cite{christiansen2016high}.

This has already been reflected in the dilemma of computational electromagnetism. On one hand, nodal Lagrange elements are desirable for their simplicity, economic degrees of freedom and point-wise evaluation; on the other hand, the $C^{0}$ vector Lagrange elements suffer from spurious modes and have difficulties in dealing with inhomogeneous materials where the normal component of the electric field may be discontinuous at interfaces \cite{boyse1992nodal}. There has been an increased interest in the use of nodal elements in computational electromagnetism.
 A common strategy is to add a stabilisation term $-\grad\div$ in the variational formulation besides the $\curl\curl$ operator (c.f.  \cite{boyse1992nodal}). In order to remove spurious modes on nonconvex domains, a weighted version of the penalty term was proposed in Costabel and Dauge \cite{costabel2002weighted}, and a local projection in the penalty term was proposed by Duan et al. \cite{duan2009local}. Eigenvalue problems approximated with  Lagrange elements are considered in 
 \cite{bonito2011approximation}. Of course this is a rather incomplete review of the literature.

From the perspective of finite element exterior calculus, $H(\curl)$ edge elements stand out for allowing normal discontinuities and fitting in a discrete sequence of spaces. Therefore we are motivated to seek nonstandard finite element differential complexes keeping these properties, but with nodal type bases. Actually, our approach is to look for elements with higher continuity on low dimensional subsimplexes (for example, extra smoothness at vertices and edges etc.). We remark that $H(\div)$ elements with nodal degrees of freedom and incomplete polynomial shape function spaces were explored in \cite{guzman2013conforming,guzman2014conforming}.

There are also some important circumstances where we need to match several copies of finite element de Rham sequences with different continuities.  
In Arnold, Falk and Winther \cite{Arnold2006a}, the authors constructed the Arnold-Winther symmetric stress element in 2D using the following Bernstein-Gelfand-Gelfand (BGG) resolution:    
 \begin{diagram}
H_{h}^{1} & \rTo^{\curl} & H_{h}(\div) & \rTo^{\div} & L_{h}^{2}& \rTo & 0\\
  & \ruTo^{S_{0}}   &  &  \ruTo^{S_{1}} &&\\
 \left (\tilde{H}_{h}^{1}\right )^{2} & \rTo^{\curl} & \left ( \tilde{H}_{h}(\div)\right )^{2}  & \rTo^{\div} &  \left (\tilde{L}_{h}^{2}\right )^{2} &\rTo & 0
 \end{diagram} 
  Here $S_{0}$ is bijective between the discrete spaces $ \left (\tilde{H}_{h}^{1}\right )^{2}$ and $H_{h}(\div)$. Usually it is not easy to find a compatible element for both $\left (H^{1}\right )^{2}$ and $H(\div) $.  Here compatibility means that this element should fit in both sequences, but for different operators. Therefore in the numerical discretisation, a compromise is to find finite element spaces such that $S_{0}^{h}:=\Pi S_{0}$ is onto, where $\Pi$ is a projection to $H_{h}(\div)$. This leads to convergent finite elements with weak symmetry \cite{arnold2007mixed}.   In order to obtain an element with strong symmetry, one has to find the isomorphic elements for $\left (\tilde{H}_{h}^{1}\right )^{2}$ and $H_{h}(\div) $. In \cite{Arnold2006a} the authors constructed such vector elements to give an explanation of Arnold-Winther elasticity elements.  

For conforming finite elements with local DoFs,  $S_{0}$ imposes stronger continuity requirement on the $H(\div)$ element, which actually leads to a Stokes complex. Therefore from the  perspective of BGG, an essential difficulty of the construction of tensor-valued elements with strong symmetry  is to construct and match discrete de Rham complexes with different continuities. 
   
On the other hand, Hu et al. \cite{hu2015family,Hu2015,hu2016finite} designed a new element for linear elasticity with strong symmetry following a different approach.  A two-step approach was proposed to design compatible elements and prove the inf-sup condition. The displacement space is divided into rigid body motions and its orthogonal complement. Rigid body motion is controlled by the face functions which are $C^{0}$ continuous, and its orthogonal complement is controlled by the bubble functions on each element, which are local. The resulting space has a canonical nodal basis. However it remains open to  understand this innovative approach and generalize the constructions to the de Rham case in the framework of finite element exterior calculus, which may yield a more systematic construction for a broader class of applications.  We also hope a systematic study could give a new perspective for the challenging problem of designing bases for high order elements.

As a summary, we have several motivations to develop the new finite element de Rham complexes in this paper:
\begin{itemize}
\item
obtaining smaller algebraic systems by elements with higher continuity, which was also the motivation of Stenberg  \cite{Stenberg2010},
\item
getting Lagrange or Hermite type nodal basis functions, which are more canonical and easier to write,
\item
making progress towards a systematic development of finite element complexes compatible with the BGG construction,  
\item
developing tools of  finite element exterior calculus to re-construct Hu-Zhang elasticity elements,
\item
a better understanding of the periodic table of finite element differential forms.
\end{itemize}

 We will introduce $r$ as a new regularity parameter in the finite element periodic table, which gives $\mathcal{P}_{r, p}\Lambda^{k}(\mathcal{T}_{h}^{n})$ on $ n$-dimensional simplicial mesh, for differential $k$ forms with piecewise polynomials of degree $p$.   Because of the periodicity, sometimes different values of $r$ in $\mathcal{P}_{r, p}\Lambda^{k}(\mathcal{T}_{h}^{n})$ may represent the same element. The major results are summarised in Table \ref{tab:2Dnotation} and Table \ref{tab:3Dnotation}.
 \begin {table}[H]
 \caption{2D families and notation  ($p\geq 2$), elements in each box are the same as the family with lowest regularity (possibly with different notation), $r=0$ and $r=1$ have the same 2D bubbles}
 \label{tab:2Dnotation}
\begin{center}
\begin{tabular}{ c |c c  c}
&$k=0$ & $k=1$ & $k=2$ \\\hline
\cline{4-4} \cline{3-3}
$r=0$&  Lagrange $\mathcal{P}_{p}\Lambda^{0}(\mathcal{T}_{h}^{2})$  &  \multicolumn{1}{|c|}{BDM  $\mathcal{P}_{p-1}\Lambda^{1}(\mathcal{T}_{h}^{2})$}  & \multicolumn{1}{|c|}{DG $\mathcal{P}_{p-2}\Lambda^{2}(\mathcal{T}_{h}^{2})$} \\
  \cline{3-3}
$r=1$&  Hermite $\mathcal{P}_{1, p+1}\Lambda^{0}(\mathcal{T}_{h}^{2})$ & Stenberg $\mathcal{P}_{1, p}\Lambda^{1}(\mathcal{T}_{h}^{2})$ &  \multicolumn{1}{|c|}{DG $\mathcal{P}_{1, p-1}\Lambda^{2}(\mathcal{T}_{h}^{2})$}  \\
  \cline{4-4}
$r=2$&  Argyris $\mathcal{P}_{2, p+3}\Lambda^{0}(\mathcal{T}_{h}^{2})$  & vector Hermite $\mathcal{P}_{2, p+2}\Lambda^{1}(\mathcal{T}_{h}^{2})$ & Falk-Neilan $\mathcal{P}_{2, p+1}\Lambda^{2}(\mathcal{T}_{h}^{2})$ \\
\end{tabular}
\end{center}
\end{table}
\begin {table}[H]
 \caption{3D families and notation  ($p\geq 3$), elements in each box are the same as the family with lowest regularity (possibly with different notation), $r=0$ and $r=1$ have the same 2D bubbles, and $r=0, 1, 2$ have the same 3D bubbles}
 \label{tab:3Dnotation}
\begin{center}
\footnotesize
\begin{tabular}{ c | c c c c }
&$k=0$ & $k=1$ & $k=2$ & $k=3$ \\\hline
\cline{4-4}\cline{3-3} \cline{5-5}
 $r=0$& Lagrange $\mathcal{P}_{p}\Lambda^{0}(\mathcal{T}_{h}^{3}) $  &  \multicolumn{1}{|c|}{N\'{e}d\'{e}lec $\mathcal{P}_{p-1}\Lambda^{1}(\mathcal{T}_{h}^{3})$}  & \multicolumn{1}{|c|}{BDM $\mathcal{P}_{p-2}\Lambda^{2}(\mathcal{T}_{h}^{3})$} & \multicolumn{1}{|c|}{DG $\mathcal{P}_{p-3}\Lambda^{3}(\mathcal{T}_{h}^{3})$} \\
  \cline{3-3}
$r=1$&  Hermite $\mathcal{P}_{1, p}\Lambda^{0}(\mathcal{T}_{h}^{3})$ & new, $\mathcal{P}_{1, p-1}\Lambda^{1}(\mathcal{T}_{h}^{3})$ &  \multicolumn{1}{|c|}{BDM $\mathcal{P}_{1, p-2}\Lambda^{2}(\mathcal{T}_{h}^{3})$}  & \multicolumn{1}{|c|}{DG $\mathcal{P}_{1, p-3}\Lambda^{3}(\mathcal{T}_{h}^{3})$} \\
  \cline{4-4}
$r=2$&  (scalar) 3D Neilan velocity $\mathcal{P}_{2, p+2}\Lambda^{0}(\mathcal{T}_{h}^{3})$ & new, $\mathcal{P}_{2, p+1}\Lambda^{1}(\mathcal{T}_{h}^{3})$ & Stenberg $\mathcal{P}_{2, p}\Lambda^{2}(\mathcal{T}_{h}^{3})$  & \multicolumn{1}{|c|}{DG $\mathcal{P}_{2, p-1}\Lambda^{3}(\mathcal{T}_{h}^{3})$} \\
  \cline{5-5}
\end{tabular}
\end{center}
 \end{table}

 For each $r=0, 1, 2$, the elements in lower dimensional spaces are restrictions of those in higher dimensions.  The $H(\div)$ elements with $n=2, r=1$ and $n=3, r=2$ coincide with the ``nonstandard $H(\div)$ elements'' of  Stenberg \cite{Stenberg2010}. 
 
The new elements have nodal degrees of freedom. The 2D $H(\curl)$ element with $r=2$ (the velocity space of the Falk-Neilan Stokes pair \cite{falk2013stokes})  consists of two copies of the scalar Hermite element. Therefore the basis of $\mathcal{P}_{2, p}\Lambda^{1}\left (\mathcal{T}_{h}^{2}\right )$ is a simple combination of the scalar Hermite bases.    The 2D $H(\div)$ element with $r=1$ (Stenberg nonstandard element \cite{Stenberg2010}) and the 3D $H(\curl)$ element with $r=2$ are essentially vectorial, i.e. they cannot be represented as copies of scalar elements. Nevertheless, the basis of these two spaces can be written based on scalar Lagrange and Hermite elements respectively by allowing tangential/normal degrees of freedom taking different values on neighboring elements. This trick cannot be applied to the 3D Stenberg $H(\div)$ element ($\mathcal{P}_{2, p}\Lambda^{2}(\mathcal{T}_{h}^{3})$) since there are no degrees of freedom on edges (except for those at vertices). We can further reduce $\mathcal{P}_{2, p}\Lambda^{2}(\mathcal{T}_{h}^{3})$ to a subspace by imposing normal continuity on edges and meanwhile retain the inf-sup condition between this space and piecewise polynomials. This space, which we will call ``Hu-Zhang type $H(\div)$ element'', is a generalization of the Hu-Zhang construction \cite{Hu2015,hu2015family} of a symmetric stress element for the Hellinger-Reissner principle of linear elasticity  and admits a Lagrange type nodal basis.
 

The elements in the new sequences are subspaces of the standard finite element de Rham complexes with complete polynomials \cite{arnold2014periodic}, which fit in our family with $r=0$.  Moreover, we can see the periodicity: 
beginning with an element with higher continuity (for example, the Hermite element for $H^{1}$ with $r=1$), the continuity decreases as we take exterior derivatives (for $r=1$, we go back to the classical BDM element for $H(\div)$ in 3D). Furthermore, we can consider restrictions to lower dimensional simplexes. This is analogous to the idea of finite element system (FES) \cite{christiansen2010finite}, and inspired us to discover the whole families.  For  $r=2$,  restriction of the finite elements to a two dimensional face has higher continuity across the one dimensional boundary of that face. In fact, this reconstructs the 2D Stokes complex of Falk and Neilan \cite{falk2013stokes}.

Although the new elements have higher continuity on low dimensional subsimplexes (vertices and edges etc.), generally these elements are not conforming approximations for higher order problems. For example, the scalar Hermite element is $C^{1}$ at vertices, but normal derivatives may be discontinuous across faces. As a result, the triangular or tetrahedral Hermite element is not a reasonable approximation for the fourth order biharmonic problem.  Since the purpose of this paper is not to pursue finite elements for high order PDEs, this will not be a trouble.  

The enriched periodic table studied in this paper also gives another possibility for  BGG constructions at least in 2D. We can use the Falk-Neilan Stokes complex ($n=2, r=2$) for the top row, and use the $n=2, r=1$ complex for the bottom. Then $S_{0}$ is an identification between vector Hermite elements.  
This leads to the 2D Hu-Zhang elasticity element \cite{hu2014family}, and explains why the stress element naturally begins with cubic polynomials  (since Hermite elements are at least cubic).

Next we recall some notation. For a contractible domain $\Omega\subset \mathbb{R}^{2}$, we have two exact sequences in 2D:
  \begin{equation}\label{sequence1}
\begin{CD}
0@>>>\mathbb{R}@>>>H(\curl; \Omega) @>\curl>> {H}(\mathrm{div}; \Omega) @>\mathrm{div} >> L^{2}(\Omega)  @ > >>  0,
\end{CD}
\end{equation}
and
 \begin{equation}\label{sequence2}
\begin{CD}
0@>>>\mathbb{R}@>>>H({\grad}; \Omega) @>{\grad}>> {H}(\mathrm{rot}; \Omega) @>\mathrm{rot} >> L^{2}(\Omega)  @ > >>  0,
\end{CD}
\end{equation}
and  we have one exact sequence in 3D on a contractible domain $\Omega\subset \mathbb{R}^{3}$:
 \begin{equation}\label{sequence3}
\begin{CD}
0@>>>\mathbb{R}@>>>H({\grad}; \Omega) @>{\grad}>> {H}(\curl; \Omega)@>\curl>> {H}(\mathrm{div}; \Omega)  @>\mathrm{div} >> L^{2}(\Omega)  @ > >>  0.
\end{CD}
\end{equation}
 One can obtain \eqref{sequence2} by rotating \eqref{sequence1} by $\pi/2$. Therefore in the remaining part of this paper, we only consider \eqref{sequence1} in 2D.

We assume that $\Omega$ is a polyhedral domain. 
In the following, we will use $\mathcal{V}$ to denote the set of vertices, $\mathcal{E}$ for the edges, $\mathcal{F}$ for the faces and $\mathcal{T}$ for the 3D cells. For a given mesh, $V$, $E$, $F$ and $T$ are used to denote the number of vertices, edges, faces and tetrahedra respectively. From Euler's formula, one has $V-E+F=1$ in 2D and $V-E+F-T=1$ in 3D for contractible domains.

We use $\bm{\nu}_{f}$ and $\bm{\tau}_{f}$ to denote the unit normal and tangential vectors of a simplex $f$ respectively. In 2D, the tangential and normal directions of an edge are uniquely defined up to an orientation. For edges in 3D there are one tangential and two normal directions, and for faces in 3D there are one normal and two tangential directions. We will write $\bm{\tau}_{e}$,  $\bm{\nu}_{e, i}$ and $\bm{\nu}_{f}$,  $\bm{\tau}_{f, i}$, $i=1, 2$ for these cases.

In our discussions, $C^{r}(\mathcal{V})$ includes functions with continuous derivatives up to order $r$ at the vertices. Similarly we can define $C^{r}(\mathcal{E})$ and $C^{r}(\mathcal{F})$ for functions with certain continuity on the edges and faces.

We use the notation $\mathcal{P}_{p}\Lambda^{k}(\Omega)$ to denote the Lagrange, second N\'{e}d\'{e}lec, BDM and discontinuous elements with polynomial degree $p$, and use $\mathcal{P}_{p}^{-}\Lambda^{k}(\Omega)$ for the family with incomplete polynomials, i.e. $\mathcal{P}_{p}^{-}\Lambda^{1}(\Omega)$ is the N\'{e}d\'{e}lec element of the first kind of degree $p$, $\mathcal{P}_{p}^{-}\Lambda^{2}(\Omega)$ is the Raviart-Thomas element of degree $p$ in 3D. We  use $\mathcal{P}_p (\Omega)$  to denote the polynomial space of degree $p$ on $\Omega$ and use $\mathcal{P}_{r, p}\Lambda^{k}(\Omega)$ to denote the families  developed below. Here $r$ is the regularity parameter and $p$ is the polynomial degree. Since we mainly consider one, two and three spatial dimensions, we  explicitly use $\grad$, $\curl$ and $\div$ instead of the exterior derivative $d$ in most cases. We  use ${\0{\mathcal{P}}}_{p}\Lambda^{k}(\Omega)$ to denote the finite element spaces with standard vanishing boundary conditions. For example, in 3D ${\0{\mathcal{P}}}_{p}\Lambda^{1}(\Omega)$ has vanishing tangential components and ${\0{\mathcal{P}}}_{p}\Lambda^{2}(\Omega)$ has vanishing normal components on $\partial \Omega$ when differential forms are represented by vector fields. Functions in ${\0{\mathcal{P}}}_{r, p}\Lambda^{k}$ may have stronger vanishing conditions on low dimensional simplexes depending on the continuity condition of $\mathcal{P}_{r, p}\Lambda^{k}$. 

We define $\ker \left ( d, V\right )$ as the kernel of the differential operator $d$ in the space $V$, i.e.
$$
\ker \left ( d, V\right ):=\{v\in V: dv=0\}.
$$

The rest of this paper is organised as follows.  In Section \ref{sec:hermite}, we construct the family with regularity parameter $r=1$ as a resolution of the Hermite element.  We verify the unisolvence and exactness of the new sequences. In Section \ref{sec:r2}, we construct the family $r=2$. In Section \ref{sec:bc}, we discuss boundary conditions.  In Section \ref{sec:geometric-decomposition}, we discuss the geometric decomposition and local exact sequences of the new complexes.
In Section \ref{sec:BGG}, we re-construct 2D Hu-Zhang elasticity elements combining BGG and the new de Rham families.
 We give concluding remarks in  Section \ref{sec:concluding}.

\section{Hermite family: $r=1$}\label{sec:hermite}

The Hermite family begins with Hermite elements in all dimensions.

\subsection{Complex in 1D} 

The 1D complex $(p\geq 2)$:
 \begin{equation}\label{1dl1}
\begin{CD}
\mathbb{R}@>>>\mathcal{P}_{1, p+1}\Lambda^{0}(\mathcal{T}_{h}^{1}) @>\mathrm{grad}>>\mathcal{P}_{1, p}\Lambda^{1}(\mathcal{T}_{h}^{1}) @>{} >>  0,
\end{CD}
\end{equation}
 consists of the $C^{1}$ Hermite element of degree $p+1$ ($\mathcal{P}_{1, p+1}\Lambda^{0}(\mathcal{T}_{h}^{1})$), and $C^{0}$ Lagrange element of degree $p$ ($\mathcal{P}_{1, p}\Lambda^{1}(\mathcal{T}_{h}^{1})$ ). It is obvious that this sequence is (globally) exact on intervals, because the gradient of the Hermite element with degree $p+1$ falls into the Lagrange element space of degree $p$, and conversely, if $\grad u_{h}={v}_{h}\in \mathcal{P}_{1, p}\Lambda^{1}(\mathcal{T}_{h}^{1})$, then $u_{h}$ is a piecewise polynomial of degree $p+1$, and has continuous first order derivatives at the vertices. This implies $u_{h}\in \mathcal{P}_{1, p+1}\Lambda^{0}(\mathcal{T}_{h}^{1}) $, i.e. $u_{h}$ belongs to the global Hermite space.

\subsection{Complex in 2D}

We describe the discrete version of sequence \eqref{sequence1}:
 \begin{equation}\label{2dl1}
\begin{CD}
\mathbb{R}@>>> \mathcal{P}_{1, p+2}\Lambda^{0}(\mathcal{T}_{h}^{2}) @>\curl>>\mathcal{P}_{1, p+1}\Lambda^{1}(\mathcal{T}_{h}^{2})   @>{\div} >> \mathcal{P}_{1, p}\Lambda^{2}(\mathcal{T}_{h}^{2}) @>{} >>  0,
\end{CD}
\end{equation}
where $p\geq 1$. For the lowest order case which starts from piecewise cubic polynomials, we show the finite element diagrams in Figure \ref{fig:div2d}.

We will use the Hermite elements to discretize  $H(\curl)$.   One can characterize the Hermite elements of degree $p$ as follows:
$$
\mathcal{P}_{1, p}\Lambda^{0}(\mathcal{T}_{h}^{2})=\{s\in H(\curl): s|_{f}\in \mathcal{P}_p , \forall f\in \mathcal{F}; s\in C^{1}(\mathcal{V}) \}.
$$

We define the  ${H}(\mathrm{div})$ finite element space of degree $p$ as
$$
\mathcal{P}_{1, p}\Lambda^{1}(\mathcal{T}_{h}^{2}):=\{\bm{v}\in {H}(\mathrm{div}): \bm{v}|_{f}\in (\mathcal{P}_p )^{2}, \forall f\in \mathcal{F};  \bm{v}\in C^{0}(\mathcal{V})\}.
$$
This was first introduced in Stenberg \cite{Stenberg2010}.

The degrees of freedom ($n=2$, $r=1$)  can be given as follows. The set of DoFs is empty when $p<0$ in $\mathcal{P}_{ p}\Lambda^{k}$.

\newcase{For $u\in \mathcal{P}_{1, p}\Lambda^{0}$:}
\begin{itemize}
\item 
function value $u(\bm{x})$ and first order derivatives $\partial_{i}u(\bm{x}), i=1, 2$ at each vertex $\bm{x}$,
\item
moments on each edge
$$
\int_{e}u\cdot q~ ds, \quad q\in \mathcal{P}_{p-4}(e), \forall e\in \mathcal{E},
$$
\item
moments on each element
$$
\int_{f}u\cdot q~ dx, \quad q\in \mathcal{P}_{p-3}(f), \forall f\in \mathcal{F}.
$$
\end{itemize}

\newcase{For $\bm{B}\in \mathcal{P}_{1, p}\Lambda^{1}$:}
\begin{itemize}
\item 
function value $\bm{B}(\bm{x})$  at each vertex $\bm{x}$,
\item
moments on each edge
$$
\int_{e}\left (\bm{B}\cdot \bm{\nu}\right )  q~ ds, \quad q\in \mathcal{P}_{p-2}(e), \forall e\in \mathcal{E},
$$
\item
moments on each element
$$
\int_{f}\bm{B}\cdot \bm{q}~ dx, \quad \bm{q}\in \mathcal{P}_{p-1}^{-}\Lambda^{1}(f), \forall f\in \mathcal{F}.
$$
\end{itemize}

\newcase{For $w\in \mathcal{P}_{1, p}\Lambda^{2}$:}
\begin{itemize}
\item
moments on each element
$$
\int_{f}w\cdot q~ dx, \quad q\in \mathcal{P}_p (f), \forall f\in \mathcal{F}.
$$
\end{itemize}

Each row (or column) of the Hu-Zhang stress element \cite{hu2015family,Hu2015} belongs to the vector valued space  $\mathcal{P}_{1, p}\Lambda^{1}(\mathcal{T}_{h}^{2})$.   Connections between
$\mathcal{P}_{1, p}\Lambda^{1}(\mathcal{T}_{h}^{2})$ and the Lagrange elements can be established based on a similar idea as the Hu-Zhang construction:   one can retain the normal degrees of freedom  on a face (edge in 2D) of the Lagrange elements,  and move the tangential  DoFs on that face (edge in 2D)  into the interior of the elements.
Alternatively we can consider decompositions of the shape function space. The space $\mathcal{P}_{1, p}\Lambda^{1}(\mathcal{T}_{h}^{2})$ can be decomposed as globally continuous Lagrange elements and   ${H}(\mathrm{div})$ bubble functions (shape functions with vanishing normal components, but the tangential components can be discontinuous). Specifically, we have the following lemma.
\begin{lemma}
We have the decomposition:
$$
\mathcal{P}_{1, p}\Lambda^{1}\left (\mathcal{T}_{h}^{2}\right )=\bm{L}_{h}^{p}+\bm{\Sigma}_{b}^{p}, \quad p\geq 2,
$$
where $\bm{L}_{h}^{p}$ is the Lagrange element of degree $p$, and the bubble function space $\bm{\Sigma}_{b}^{p}$ is defined as
$$
\bm{\Sigma}_{b}^{p}=\{ \bm{B}_{h}\in  \mathcal{P}_{1, p}\Lambda^{1}\left (\mathcal{T}_{h}^{2}\right ): \bm{B}_{h}\cdot \bm{\nu}_{e}=0, \forall e\in \mathcal{E}\},
$$
where $\bm{\nu}_{e}$ is the normal direction of $e$.
\end{lemma}
\begin{proof}
It is obvious that $\bm{L}_{h}^{p}+\bm{\Sigma}_{b}^{p}\subset \mathcal{P}_{1, p}\Lambda^{1}(\mathcal{T}_{h}^{2})$. Conversely,  given $\bm{u}_{h}\in \mathcal{P}_{1, p}\Lambda^{1}(\mathcal{T}_{h}^{2})$,  define $\tilde{\bm{u}}_{h}\in \bm{L}_{h}^{p}$ by specifying its DoFs: the vertex DoFs and the normal DoFs on the edges are defined to be the same as $\bm{u}_{h}$, while other DoFs are defined to be zero. By definition $\bm{u}_{h}-\tilde{\bm{u}}_{h}$ has vanishing normal DoFs, and according to the conforming property, we know that the normal components of $\bm{u}_{h}-\tilde{\bm{u}}_{h}$ vanish. Therefore $\bm{u}_{h}-\tilde{\bm{u}}_{h}\in \bm{\Sigma}_{b}^{p}$.
\end{proof}
\begin{center}
\begin{figure}

\setlength{\unitlength}{1.2cm}
\begin{picture}(2,2)(-0.8,0)
\put(0,0){
\begin{picture}(2,2)
\put(-1, 0){\line(1,2){1}} 
\put(0, 2){\line(1,-2){1}}
\put(-1,0){\line(1,0){2}}
\put(-1.,0){\circle*{0.1}}
\put(-1.,0){\circle{0.2}}
\put(1.,0){\circle*{0.1}}
\put(1.,0){\circle{0.2}}
\put(0,2){\circle*{0.1}}
\put(0,2){\circle{0.2}}
\put(0, 0.6){\circle*{0.1}}
\end{picture}
}

\put(1.5, 1){\vector(1, 0){1}}
\put(1.68, 1.15){$\curl$}

\put(4,0){
\put(-1, 0){\line(1,2){1}} 
\put(0, 2){\line(1,-2){1}}
\put(-1,0){\line(1,0){2}}
\put(-0.05,2){\circle*{0.1}}
\put(0.05,2){\circle*{0.1}}
\put(-1.05,0){\circle*{0.1}}
\put(-0.95,0){\circle*{0.1}}
\put(-1.05,0){\circle*{0.1}}
\put(-0.95,0){\circle*{0.1}}
\put(1.05,0){\circle*{0.1}}
\put(0.95,0){\circle*{0.1}}
\put(0.5, 1){\vector(2,1){0.3}}
\put(-0.5, 1){\vector(-2,1){0.3}}
\put(0, 0){\vector(0,-1){0.3}}
\put(0, 0.8){\circle*{0.1}}
\put(-0.1, 0.6){\circle*{0.1}}
\put(0.1, 0.6){\circle*{0.1}}
}

\put(5.5, 1){\vector(1, 0){1}}
\put(5.75, 1.1){{$\div$}}

\put(8,0){
\begin{picture}(2,2)
\put(-1, 0){\line(1,2){1}} 
\put(0, 2){\line(1,-2){1}}
\put(-1,0){\line(1,0){2}}
\put(0, 0.8){\circle*{0.1}}
\put(-0.1, 0.6){\circle*{0.1}}
\put(0.1, 0.6){\circle*{0.1}}
\end{picture}
}

\end{picture}
\caption{ Finite element sequence of lowest order (2D, r=1): $H(\curl)\rightarrow {H}(\div)\rightarrow L^{2}$ with the local shape function spaces $\mathcal{P}_{3}\rightarrow (\mathcal{P}_{2})^{2}\rightarrow \mathcal{P}_{1}$. }
\label{fig:div2d}
\end{figure}
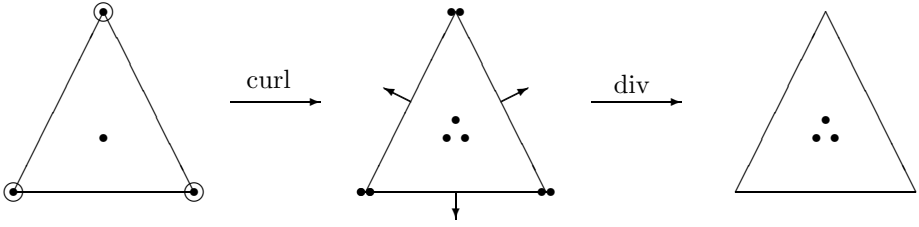
\end{center}


\begin{theorem}
The sequence \eqref{2dl1} is a complex, which is exact on contractible domains.
\end{theorem}
\begin{proof}
From the definitions of $\mathcal{P}_{1, p+2}\Lambda^{0}(\mathcal{T}_{h}^{2})$, $ \mathcal{P}_{1, p+1}\Lambda^{1}(\mathcal{T}_{h}^{2})$ and $\mathcal{P}_{1, p}\Lambda^{2}(\mathcal{T}_{h}^{2})$, it is obvious that $\curl \mathcal{P}_{1, p+2}\Lambda^{0}(\mathcal{T}_{h}^{2})\subset  \mathcal{P}_{1, p+1}\Lambda^{1}(\mathcal{T}_{h}^{2})$ and $\mathrm{div} \mathcal{P}_{1, p+1}\Lambda^{1}(\mathcal{T}_{h}^{2})\subset \mathcal{P}_{1, p}\Lambda^{2}(\mathcal{T}_{h}^{2})$.  

Next we show that \eqref{2dl1} is exact.  From the inf-sup condition proved in \cite{Stenberg2010},  we know that the $\mathrm{div}$ operator is onto, i.e. $\div  \mathcal{P}_{1, p+1}\Lambda^{1}(\mathcal{T}_{h}^{2})=\mathcal{P}_{1, p}\Lambda^{2}(\mathcal{T}_{h}^{2})$. Therefore it suffices to count the dimensions.  
We note that the global dimension of  Hermite element of degree $p+2$ is ${\dim}(\mathcal{P}_{1, p+2}\Lambda^{0}(\mathcal{T}_{h}^{2}))=  3V+(p-1)E+1/2p(p+1)F$, and the dimensions of $ \mathcal{P}_{1, p+1}\Lambda^{1}(\mathcal{T}_{h}^{2})$ and $\mathcal{P}_{1, p}\Lambda^{2}(\mathcal{T}_{h}^{2})$ are ${\dim}( \mathcal{P}_{1, p+1}\Lambda^{1}(\mathcal{T}_{h}^{2}))=2V+pE+(2{p+3 \choose 2}  -3p-6)F$ and ${\dim}(\mathcal{P}_{1, p}\Lambda^{2}(\mathcal{T}_{h}^{2}))={p+2 \choose 2}F$ respectively.  By straightforward calculations, we have 
$$
 \dim\left (\mathcal{P}_{1, p+1}\Lambda^{1}(\mathcal{T}_{h}^{2})\right )= \left [{\dim}(\mathcal{P}_{1, p+2}\Lambda^{0}(\mathcal{T}_{h}^{2}))-1\right ]+ {\dim}(\mathcal{P}_{1, p}\Lambda^{2}(\mathcal{T}_{h}^{2})).
$$
\end{proof}

By rotating the elements in sequence \eqref{2dl1}, we can get another $\grad$-$\rot$ finite element complex.

\paragraph{Basis function.}


The basis functions of $ \mathcal{P}_{1, p+1}\Lambda^{1}(\mathcal{T}_{h}^{2})$  can be written in a similar way as the Lagrange basis. For example, for $\mathcal{P}_{1, p+1}\Lambda^{1}(\mathcal{T}_{h}^{2})$ with continuous normal components, we write the two basis functions associated to an edge Lagrange point as one normal basis and one tangential basis. We require each normal basis  to be single-valued in the two elements sharing the edge, while we allow a tangential basis function  taking different values in the two neighbour elements. 

We explicitly construct a basis  of $ \mathcal{P}_{1, p}\Lambda^{1}(\mathcal{T}_{h}^{2})$. Below we will use $\phi_{\bm{x}}^{p}$ to denote the nodal basis function of the Lagrange elements at a Lagrange interpolation point $\bm{x}$. The superscript $p$ in $\phi_{\bm{x}}^{p}$ indicates that it is a polynomial of degree $p$. For simplicity of presentation, we omit this superscript below when there is no possible confusion, i.e. we will write $\phi_{\bm{x}}$ instead.
We will use $\bm{e}_{i}$ to denote the canonical basis $(0, \cdots, 1, \cdots, 0)$ in the Euclidean space $\mathbb{R}^{n}$.

Basis functions of the ${H}(\mathrm{div})$ finite element space $ \mathcal{P}_{1, p}\Lambda^{1}(\mathcal{T}_{h}^{2})$ can be constructed as:
\begin{enumerate}
\item
vertex-based basis functions: given $\bm{x}\in \mathcal{V}$, its two basis functions are 
$$
\bm{v}_{\bm{x}, i}=\phi_{\bm{x}}\bm{e}_{i}, \quad i=1, 2,
$$
\item
edge-based basis functions: given a Lagrange point $\bm{x}$ on an edge $e$, its associated basis function with the normal direction:
$$
\bm{v}_{e, \bm{x}}= \phi_{\bm{x}}\bm{\nu}_{e},
$$
where $\bm{\nu}_{e}$ is the normal vector of the edge $e$,
\item
edge-based basis functions: given Lagrange point $\bm{x}$ on an edge $e$, its associated basis functions with the tangential direction:
$$
\bm{v}_{e, \bm{x}, i}= \phi_{\bm{x}}|_{f_{i}}\bm{\tau}_{e}, \quad i=1,2,
$$
where $f_{1}$ and $f_{2}$ are the two elements sharing the edge $e$, $\bm{\tau}_{e}$ is the tangential vector of the edge $e$,
\item
interior basis functions: at an interior Lagrange point $\bm{x}$, its two associated basis functions: 
$$
\bm{v}_{f, \bm{x}, i}=\phi_{\bm{x}}\bm{e}_{i}, \quad i=1,2.
$$
\end{enumerate}

\subsection{Complex in 3D}

We now turn to the 3D complexes.  For $p\geq 0$ we formally write the sequence as
{\footnotesize
 \begin{equation}\label{3dl1}
\begin{CD}
\mathbb{R}@>>>\mathcal{P}_{1, p+3}\Lambda^{0}(\mathcal{T}_{h}^{3}) @>\grad>>\mathcal{P}_{1, p+2}\Lambda^{1}(\mathcal{T}_{h}^{3})  @>{\curl}
>>  \mathcal{P}_{1, p+1}\Lambda^{2}(\mathcal{T}_{h}^{3})      @>\div>>    \mathcal{P}_{1, p}\Lambda^{3}(\mathcal{T}_{h}^{3}) @>{} >>  0.
\end{CD}
\end{equation}
}

Here $\mathcal{P}_{1, p+3}\Lambda^{0}(\mathcal{T}_{h}^{3}) $ is the Hermite finite element in 3D with polynomial degree $p+3$, and $\mathcal{P}_{1, p+2}\Lambda^{1}(\mathcal{T}_{h}^{3})$ is a Stenberg-type $H(\curl)$ space, which is tangentially continuous on edges and faces and $C^{0}$ continuous at the vertices.  To our knowledge, this element is new in the literature.

We can give the following DoFs for $\mathcal{P}_{1, p}\Lambda^{1}(\mathcal{T}_{h}^{3})$ (as shown in Figure \ref{fig:3dr1}):
\begin{enumerate}
\item
function values $\bm{u}(\bm{x})$ at each vertex $\bm{x}\in \mathcal{V}$,
\item
$p-1$ tangential DoFs on each edge $e$: 
$$
\int_{e}\left (\bm{u}\cdot \bm{\tau}_{e}\right ) {q}, \quad\forall {q}\in \mathcal{P}_{p-2}\Lambda^{0}(e).
$$ 
\item
DoFs on each face $f$: 
$$
\int_{f} \left ( \bm{u}\times \bm{\nu}_{f}\right ) \cdot  \bm{\omega}, \quad \forall \bm{\omega}\in \mathcal{P}_{p-1}^{-}\Lambda^{1}(f),
$$
where $\bm{u}\times \bm{\nu}_{f}$ is understood as a two dimensional vector on face $f$, 
\item
interior DoFs on each 3D cell $t$: 
$$
\int_{t} \bm{u} \cdot \bm{\eta}, \quad \forall \bm{\eta}\in \mathcal{P}_{p-2}^{-}\Lambda^{2}(t).
$$
\end{enumerate}

\begin{center}
\begin{figure}
\setlength{\unitlength}{1.8mm}
\begin{picture}(20,20)(0, 0)
\put(36,0){
\put(-10, 0){\line(1,2){10}} 
\put(0, 20){\line(1,-2){10}}
\put(-10,0){\line(1,0){20}}
\multiput(-10,0)(1,0.6){10}{\circle*{0.01}}
\multiput(10,0)(-1,0.6){10}{\circle*{0.01}}
\multiput(0,6)(0,1.3){10}{\circle*{0.01}}

\put(-10.5,-0.5){\circle*{0.6}}
\put(-9.5,-0.5){\circle*{0.6}}
\put(-10,0.5){\circle*{0.6}}
\put(10.5,-0.5){\circle*{0.6}}
\put(9.5,-0.5){\circle*{0.6}}
\put(10,0.5){\circle*{0.6}}
\put(0.5,19.5){\circle*{0.6}}
\put(-0.5,19.5){\circle*{0.6}}
\put(0,20.5){\circle*{0.6}}
\put(0.5, 5.5){\circle*{0.6}}
\put(-0.5, 5.5){\circle*{0.6}}
\put(0, 6.5){\circle*{0.6}}

\put(-5, 6){+3}
\put(3, 6){+3}
\put(0, 2){+3}
\put(-3, 8){+3}

{\linethickness{1pt}
\put(5,10){\vector(1,-2){1.8}}
\put(-5,10){\vector(-1, -2){1.8}}
\put(0,11){\vector(0,-2){2.2}}
\put(5,3){\vector(3,-2){2.3}}
\put(-5,3){\vector(-3,-2){2.3}}
\put(0,0){\vector(1,0){2.3}}
}
}
\end{picture}
\caption{Lowest order ($\mathcal{P}_{2}$) $H(\curl)$ element of continuity $r=1$.}
\label{fig:3dr1}
\end{figure}
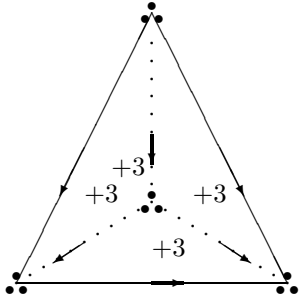
\end{center}

In order to prove the unisolvence, we need the following two results which can be found, for example,  in Arnold, Falk and Winther \cite{Arnold.D;Falk.R;Winther.R.2006a} (Lemma 4.7).
\begin{lemma}\label{interDOF}
Let $\omega\in {\0{\mathcal{P}}}_{p}\Lambda^{k}(t)$. If 
$$
\int_{t}\omega\wedge \eta=0, \quad \forall \eta\in \mathcal{P}_{p-n+k}^{-}\Lambda^{n-k}(t),
$$
where $n$ is the dimension of $t$, we have $\omega=0$.
\end{lemma}

\begin{lemma}\label{lem:dimPminus} 
(c.f. \cite{Arnold.D;Falk.R;Winther.R.2006a} (3.15)) We have the dimension count:
$$
\mathrm{dim}\mathcal{P}_{p}^{-}\Lambda^{k}(\mathbb{R}^{n})={k+p-1 \choose k}{n+p \choose n-k}.
$$
\end{lemma}

We now state the unisolvence results:
\begin{lemma}\label{lem:unisol-2dr1}
The DoFs of $\mathcal{P}_{1, p}\Lambda^{1}(\mathcal{T}_{h}^{3})  $ are unisolvent.
\end{lemma}
\begin{proof}
First we check the dimension. From Lemma \ref{lem:dimPminus}, we see
$$
\mathrm{dim}\left (\mathcal{P}_{p-1}^{-}\Lambda^{1}(f)\right )={p-1 \choose 1}{p+1 \choose 1}=(p-1)(p+1),
$$
and
$$
\mathrm{dim}\left ( \mathcal{P}_{p-2}^{-}\Lambda^{2}(t)\right )={p-1\choose 2}{p+1\choose 1}=\frac{1}{2}(p-2)(p-1)(p+1).
$$
Therefore the dimension of DoFs is $3V+(p-1)E+(p-1)(p+1)F+ 1/2(p-2)(p-1)(p+1)T$. On one element ($V=4, E=6, F=4, T=1$), this amounts to 
$$
3\times 4+6(p-1)+4(p-1)(p+1)+1/2(p-2)(p-1)(p+1)=\frac{1}{2}p^{3}+3p^{2}+\frac{11}{2}p+3,
$$
which is the same as the dimension of the space of 3D polynomials of degree $p$:
$$
\mathrm{dim}\left (\mathcal{P}_{p}(t)^{3}\right )=3\cdot {p+3\choose 3}=\frac{1}{2}(p^{3}+6p^{2}+11p+6).
$$
Then it suffices to show that $\bm{u}=0$ if all the DoFs are zero.

From the vertex and edge DoFs,  it is obvious that $\bm{u}\cdot\bm{\tau}_{e}=0$ on all the edges. Then combining Lemma \ref{interDOF} with the definitions of the face DoFs, we have $\bm{u}\times \bm{\nu}_{f}=\bm{0}, ~~\forall f\in \mathcal{F}$. Finally from the interior DoFs, we have $\bm{u}=0$.

This proves the unisolvence.
\end{proof}

At $   \mathcal{P}_{1, p+1}\Lambda^{2}(\mathcal{T}_{h}^{3})  $ the new complex branches into the standard finite element de Rham sequence: $   \mathcal{P}_{1, p+1}\Lambda^{2}(\mathcal{T}_{h}^{3})  $ is the BDM space with polynomial degree $p+1$, and $   \mathcal{P}_{1, p}\Lambda^{3}(\mathcal{T}_{h}^{3})  $ is the   space of piecewise polynomials of degree $p$.

\begin{lemma}
The 3D complex of $r=1$ \eqref{3dl1} is locally and globally exact on contractible domains.
\end{lemma}
\begin{proof}
The local exactness on an element only involves properties of local polynomials, which is well known. We only show the global exactness.

The exactness at $  \mathcal{P}_{1, p+3}\Lambda^{0}(\mathcal{T}_{h}^{3})$ is trivial because the kernel of the $\grad$ operator only consists of constant functions.  It is well known that 
$$
\div:   \mathcal{P}_{1, p+1}\Lambda^{2}(\mathcal{T}_{h}^{3}) \rightarrow   \mathcal{P}_{1, p}\Lambda^{3}(\mathcal{T}_{h}^{3})
$$
for the BDM-DG pair is onto.  This proves the exactness at $  \mathcal{P}_{1, p}\Lambda^{3}(\mathcal{T}_{h}^{3}) $.  Furthermore, we note that 
$$
  \mathcal{P}_{1, p+2}\Lambda^{1}(\mathcal{T}_{h}^{3}) = \mathcal{P}_{p+2}\Lambda^{1}(\mathcal{T}_{h}^{3})\cap \{\bm{E}: \bm{E}\in C^{0}(\mathcal{V})\}.
$$
From the standard results, we have
$$
\ker\left ({\curl},  \mathcal{P}_{p+2}\Lambda^{1}(\mathcal{T}_{h}^{3})\right )=\grad  \mathcal{P}_{p+3}\Lambda^{0}(\mathcal{T}_{h}^{3}),
$$
which implies that
\begin{align*}
\ker \left ({\curl}, \mathcal{P}_{1, p+2}\Lambda^{1}(\mathcal{T}_{h}^{3})\right ) &=\grad \left ( \mathcal{P}_{p+3}\Lambda^{0}(\mathcal{T}_{h}^{3}) \cap \{u: u\in C^{1}(\mathcal{V})\}   \right )\\
&=\grad \mathcal{P}_{1, p+3}\Lambda^{0}(\mathcal{T}_{h}^{3}),
\end{align*}
which shows the exactness at $\mathcal{P}_{1, p+2}\Lambda^{1}(\mathcal{T}_{h}^{3})$.

Then we only need to show the exactness at $\mathcal{P}_{1, p+1}\Lambda^{2}(\mathcal{T}_{h}^{3})$. After verifying the exactness at all the other spaces, we can check the dimensions to show the desired results.

The Hermite element of degree $p+3$, i.e. $\mathcal{P}_{1, p+3}\Lambda^{0}(\mathcal{T}_{h}^{3})$, has dimension $4V+pE+1/2(p+2)(p+1)F+1/6p(p+1)(p+2)T$, and the space $ \mathcal{P}_{1, p+2}\Lambda^{1}(\mathcal{T}_{h}^{3})$ has dimension $3V+(p+1)E+(p+1)(p+3)F+ 1/2p(p+1)(p+3)T$.
Furthermore,
\begin{align*}
\mathrm{dim}\left( \mathcal{P}_{p+1}\Lambda^{2}(t) \right)&=\mathrm{dim}\left (\mathcal{P}_{p+1}^{-}\Lambda^{0}  \right )F+\mathrm{dim}\left ( \mathcal{P}_{p}^{-}\Lambda^{1}  \right )T\\
&={p\choose 0}{p+3\choose 2}F+{p\choose 1}{p+3\choose 2}T\\
&= \frac{1}{2}(p+2)(p+3)F+\frac{1}{2}p(p+2)(p+3)T,
\end{align*}
and $\mathrm{dim}\left ( \mathcal{P}_{p}\Lambda^{3}(t)  \right )={p+3\choose 3}=1/6(p+3)(p+2)(p+1)T$.

Checking the dimensions and using Euler's formula, we have proved the exactness.
\end{proof}

\section{Argyris family: $r=2$}\label{sec:r2}

The $r=2$ family starts with elements with $C^{2}$ continuity at vertices, $C^{1}$ on edges and $C^{0}$ on faces. In 1D this leads to $H^{2} \rightarrow H^{1}$ pairs. In 2D, this gives $H^{2}\rightarrow H^{1} \rightarrow L^{2}$ conforming elements. In 3D we obtain conforming discretisations of $H^{1} \rightarrow H(\curl) \rightarrow H(\div) \rightarrow L^{2}$.

\subsection{Complexes in 1D and 2D}


In 1D the sequence consists of the Argyris-Hermite pair as the name suggests. For the lowest polynomial degree, we have the $\mathcal{P}_{5}$-$\mathcal{P}_{4}$ pair. We actually obtain a conforming finite element sub-complex of
  \begin{equation}
\begin{CD}
\mathbb{R}@>>>H^{2}(\Omega)@>\mathrm{grad}>>H^{1}(\Omega) @>{} >>  0.
\end{CD}
\end{equation}


The 2D sequence coincides with the Stokes complex given in Falk and Neilan \cite{falk2013stokes}. Because of the higher regularity at vertices and on edges, here the sequence with $d=2$, $r=2$ leads to a conforming discretisation of the Stokes complex
  \begin{equation}
\begin{CD}
\mathbb{R}@>>>H^{2}(\Omega)@>{\curl}>>  H^{1}(\Omega)^{2}@>\mathrm{div}>>  L^{2}(\Omega) @>{} >>  0.
\end{CD}
\end{equation}
Falk and Neilan \cite{falk2013stokes} chose the Argyris element for $H^{2}(\Omega)$, the Hermite element for each components of $H^{1}(\Omega)^{2}$ and the element with $C^{0}$ continuity at the vertices for $ L^{2}(\Omega)$. The inf-sup conditions and exactness were also shown in \cite{falk2013stokes}.

\subsection{Complex in 3D}

In 3D the ${H}(\curl)$ space is a  generalisation of the Hermite element by allowing jumps in the normal directions.
In what follows we construct the finite elements. The sequence of the lowest order elements is shown in Figure \ref{fig:3d}.
\begin{center}
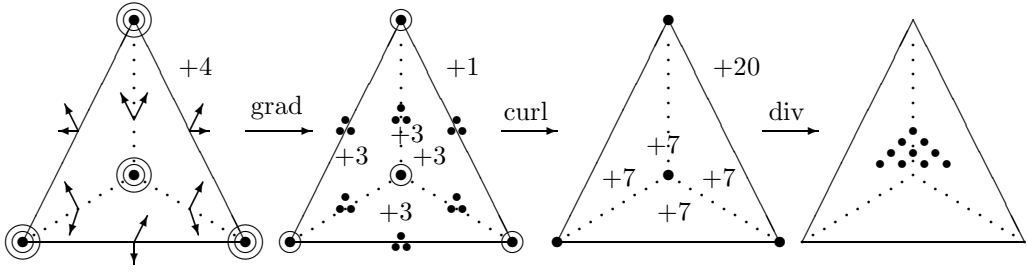
\begin{figure}
\setlength{\unitlength}{1.48mm}

\begin{picture}(20,20)(-11, 0)
\put(0,0){
\put(-10, 0){\line(1,2){10}} 
\put(0, 20){\line(1,-2){10}}
\put(-10,0){\line(1,0){20}}
\multiput(-10,0)(1,0.6){10}{\circle*{0.01}}
\multiput(10,0)(-1,0.6){10}{\circle*{0.01}}
\multiput(0,6)(0,1.3){10}{\circle*{0.01}}

\put(-10.,0){\circle*{1}}
\put(-10.,0){\circle{2}}
\put(-10.,0){\circle{3}}
\put(10.,0){\circle*{1}}
\put(10.,0){\circle{2}}
\put(10.,0){\circle{3}}
\put(0,20){\circle*{1}}
\put(0,20){\circle{2}}
\put(0,20){\circle{3}}
\put(0, 6){\circle*{1}}
\put(0, 6){\circle{2}}
\put(0, 6){\circle{3}}

\put(4,15){+4}
\put(5,10){\vector(1,2){1.2}}
\put(5,10){\vector(2,0){1.8}}
\put(-5,10){\vector(-1,2){1.2}}
\put(-5,10){\vector(-2,0){1.8}}
\put(0,11){\vector(1,2){1.2}}
\put(0,11){\vector(-1,2){1.2}}
\put(5,3){\vector(1,2){1.2}}
\put(5,3){\vector(1,-3){0.8}}
\put(-5,3){\vector(-1,2){1.2}}
\put(-5,3){\vector(-1,-3){0.8}}
\put(0,0){\vector(1,2){1.2}}
\put(0,0){\vector(0,-1){2}}
}
\put(10, 10){\vector(1, 0){6}}
\put(10.5, 11.5){$\grad$}


\put(24,0){
\put(-10, 0){\line(1,2){10}} 
\put(0, 20){\line(1,-2){10}}
\put(-10,0){\line(1,0){20}}
\multiput(-10,0)(1,0.6){10}{\circle*{0.01}}
\multiput(10,0)(-1,0.6){10}{\circle*{0.01}}
\multiput(0,6)(0,1.3){10}{\circle*{0.01}}

\put(-10.,0){\circle*{1}}
\put(-10.,0){\circle{2}}
\put(10.,0){\circle*{1}}
\put(10.,0){\circle{2}}
\put(0,20){\circle*{1}}
\put(0,20){\circle{2}}
\put(0, 6){\circle*{1}}
\put(0, 6){\circle{2}}

\put(5,11){\circle*{0.8}}
\put(4.5,10){\circle*{0.8}}
\put(5.5,10){\circle*{0.8}}
\put(-5,11){\circle*{0.8}}
\put(-4.5,10){\circle*{0.8}}
\put(-5.5,10){\circle*{0.8}}
\put(0,12){\circle*{0.8}}
\put(-0.5,11){\circle*{0.8}}
\put(0.5,11){\circle*{0.8}}

\put(5,4){\circle*{0.8}}
\put(4.5,3){\circle*{0.8}}
\put(5.5,3){\circle*{0.8}}
\put(-5,4){\circle*{0.8}}
\put(-4.5,3){\circle*{0.8}}
\put(-5.5,3){\circle*{0.8}}

\put(0,0.5){\circle*{0.8}}
\put(-0.5,-0.5){\circle*{0.8}}
\put(0.5,-0.5){\circle*{0.8}}

\put(-6, 7){+3}
\put(1, 7){+3}
\put(-2, 2){+3}
\put(-1, 9){+3}

\put(4, 15){+1}
}

\put(33, 10){\vector(1, 0){5}}
\put(33.2, 11){$\curl$}

\put(48,0){
\put(-10, 0){\line(1,2){10}} 
\put(0, 20){\line(1,-2){10}}
\put(-10,0){\line(1,0){20}}
\multiput(-10,0)(1,0.6){10}{\circle*{0.01}}
\multiput(10,0)(-1,0.6){10}{\circle*{0.01}}
\multiput(0,6)(0,1.3){10}{\circle*{0.01}}

\put(-10.,0){\circle*{1}}
\put(10.,0){\circle*{1}}
\put(0,20){\circle*{1}}
\put(0, 6){\circle*{1}}

\put(-6, 5){+7}
\put(3, 5){+7}
\put(-2, 8){+7}
\put(-1, 2){+7}
\put(4, 15){+20}
}
\put(56.5, 10){\vector(1, 0){5}}
\put(57, 11){$\div$}

\put(70,0){
\put(-10, 0){\line(1,2){10}} 
\put(0, 20){\line(1,-2){10}}
\put(-10,0){\line(1,0){20}}
\multiput(-10,0)(1,0.6){10}{\circle*{0.01}}
\multiput(10,0)(-1,0.6){10}{\circle*{0.01}}
\multiput(0,6)(0,1.3){10}{\circle*{0.01}}

\put(0,10){\circle*{0.8}}
\put(-1,9){\circle*{0.8}}
\put(1,9){\circle*{0.8}}
\put(-2,8){\circle*{0.8}}
\put(0,8){\circle*{0.8}}
\put(2,8){\circle*{0.8}}
\put(-3,7){\circle*{0.8}}
\put(-1, 7){\circle*{0.8}}
\put(3,7){\circle*{0.8}}
\put(1, 7){\circle*{0.8}}
}
\end{picture}
\caption{3D $r=2$ finite element sequence with lowest polynomial degrees $\mathcal{P}_{5}\rightarrow (\mathcal{P}_{4})^{3} \rightarrow (\mathcal{P}_{3})^{3} \rightarrow \mathcal{P}_{2}$. Interior DoFs (except for the last space) are shown as ``+4, +1, +20''.  For the $H(\curl)$ and the $H(\div)$ elements, the DoFs of the three components at each vertex are shown by one circle.
}
\label{fig:3d}
\end{figure}
\end{center} 

\newcase{Element $\mathcal{P}_{2, p}\Lambda^{0}(\mathcal{T}_{h}^{3}) \subset H^{1}(\Omega)$.}
We use $\mathcal{P}_{2, p}\Lambda^{0}(\mathcal{T}_{h}^{3})$ to denote the  finite element subspace of $H^{1}(\Omega)$ consisting of polynomials of degree $p$, which coincides with each component of  the velocity space in the 3D Stokes complex of Neilan \cite{Neilan2015}.
For  $\mathcal{P}_{2, p}\Lambda^{0}(\mathcal{T}_{h}^{3})$ we  impose $C^{2}$ continuity at vertices, $C^{1}$ on edges and $C^{0}$ on faces, i.e. 
$$
\mathcal{P}_{2, p}\Lambda^{0}(\mathcal{T}_{h}^{3})=\{ s\in H^{1}(\Omega): s|_{t}\in \mathcal{P}_p (t), \forall t\in \mathcal{T}; s\in C^{2}(\V), s\in C^{1}(\E)\}.
$$
Furthermore, the restriction of $\mathcal{P}_{2, p}\Lambda^{0}(\mathcal{T}_{h}^{3})$ to a face is a 2D Argyris element.

The DoFs are given in (3.2) of \cite{Neilan2015}. The dimensions can be counted as (for $p \geq 5$):
\begin{itemize}
\item one function value and (nine) derivatives up to second order at each vertex,
\item
$2(p-4)$ normal derivatives and $p-5$ function values on each edge, 
\item
${p-4 \choose 2}$ DoFs on each face, 
\item
$ {p-1 \choose 3} $ interior DoFs. 
\end{itemize}

The proof of the unisolvence can also be found in \cite{Neilan2015}.
\begin{lemma}
The DoFs for $\mathcal{P}_{2, p}\Lambda^{0}(\mathcal{T}_{h}^{3})$ are locally unisolvent, and $\mathcal{P}_{2, p}\Lambda^{0}(\mathcal{T}_{h}^{3})\subset H^{1}(\Omega)$.
\end{lemma}

We count the dimensions of the global finite element space:
$$
{\dim}(\mathcal{P}_{2, p}\Lambda^{0}(\mathcal{T}_{h}^{3}))=10V+\left [ 2(p-4)+(p-5)  \right ]E+{p-4 \choose 2}  F+{p-1\choose 3} T.
$$

\newcase{Element $\mathcal{P}_{2, p}\Lambda^{1}(\mathcal{T}_{h}^{3})\subset H(\curl)$.}
The $H(\curl)$ finite element space $\mathcal{P}_{2, p}\Lambda^{1}(\mathcal{T}_{h}^{3})$ is partly motivated by the Hu-Zhang elements \cite{Hu2015} for the Hellinger-Reissner variational principle of linear elasticity, which are modifications of the nodal Lagrange elements. Here we modify the Hermite elements to give a discretisation of $\mathcal{P}_{2, p}\Lambda^{1}(\mathcal{T}_{h}^{3})$.

We can describe $\mathcal{P}_{2, p}\Lambda^{1}(\mathcal{T}_{h}^{3})$ by the local shape function space and the interelement continuity:
$$
\mathcal{P}_{2, p}\Lambda^{1}(\mathcal{T}_{h}^{3})=\{ \bm{w}\in H(\curl; \Omega): \forall t\in \T, \bm{w}|_{t}\in \mathcal{P}_p \Lambda^{1}(t); \bm{w}\in C^{1}(\V), \bm{w}\in C^{0}(\E)\}.
$$
The local DoFs can be given as:
\begin{itemize}
\item
function value and first order derivatives of each component at each vertex $\bm{x}\in \mathcal{V}$:
$$
\bm{E}_{i}(\bm{x}), \quad \partial_{j}\bm{E}_{i}(\bm{x}), \quad i, j=1, 2, 3,
$$
\item
$p-3$ moments  for each component  on each edge $e\in \mathcal{E}$:
$$
\int_{e} \bm{E}_{i}{q}, \quad {q}\in \mathcal{P}_{p-4}(e), ~i=1,2, 3,
$$
\item
moments of tangential components  on each face $f$:
$$
\int_{f}\left (\bm{E}\times \bm{\nu}_{f}\right )\cdot \bm{q}, \quad \bm{q}\in \left (\mathcal{P}_{p-3}(f)\right )^{2},
$$
where $\bm{E}\times \bm{\nu}_{f}$ is considered as a 2D vector on $f$,
\item
interior DoFs on $t\in \mathcal{T}$:
\begin{align}\label{curl-bubble-r2}
\int_{t}\bm{E}\cdot \bm{v}, \quad \bm{v}\in \mathcal{P}^{-}_{p-2}\Lambda^{2}(t).
\end{align}
\end{itemize}

The dimension of the bubble space on $t\in \mathcal{T}$ is $1/2(p^{3}-2p^{2}-p+2)$.

We can immediately check the local unisolvence:
\begin{lemma}
The DoFs for $\mathcal{P}_{2, p}\Lambda^{1}(\mathcal{T}_{h}^{3})$  are unisolvent. 
\end{lemma}
\begin{proof}
It is straightforward to check the local dimension of the DoFs on an element $t$:
$$
12\times 4+3(p-3)\times 6+2{p-2 \choose 2} \times 4+\left (\frac{1}{2}p^{3}-p^{2}-\frac{1}{2}p+1\right)\times 1=\mathrm{dim}\left (\mathcal{P}_p (t)^{3}\right ).
$$

Now it suffices to show that if all the DoFs vanish,  we have $\bm{E}=0$ on the element $t$. Actually, from the DoFs attached to 调 the vertices and edges, we know that $\bm{E}$ vanishes on all the edges.  By the DoFs on faces, the tangential components of $\bm{E}$ vanish  on all the faces, therefore $\bm{E}\in {\0{\mathcal{P}}}_p \Lambda^{1}(t)$.  Finally, from the interior DoFs \eqref{curl-bubble-r2} and Lemma \ref{interDOF}, Lemma \ref{lem:dimPminus}, we have $\bm{E}=\bm{0}$, which shows the unisolvence.
\end{proof}

The dimension of $\mathcal{P}_{2, p}\Lambda^{1}(\mathcal{T}_{h}^{3})$ is 
$$
\mathrm{dim}(\mathcal{P}_{2, p}\Lambda^{1}(\mathcal{T}_{h}^{3}))=12V+3(p-3)E+2{p-2 \choose 2} F+\left(\frac{1}{2}p^{3}-p^{2}-\frac{1}{2}p+1\right)T.
$$

We characterise $H(\curl)$ bubbles in a more constructive way, which resembles the discussions for the symmetric matrix valued $H(\div, \mathbb{S})$ bubble function in \cite{hu2015family,Hu2015}.

We define the ${H}(\curl)$ bubble space on a 3D cell $t$:
$$
\Sigma_{t, p}^{c}:=\sum_{i=0}^{3} \mathcal{P}_{p-3}(t) \lambda_{j}\lambda_{l}\lambda_{m}\bm{\nu}_{i},
$$
where   $i, j, l, m$ are the four different indices from $0$ to $3$, $\lambda_{j}$ is the $j$-th barycentric coordinate and $\bm{\nu}_{i}$ is the normal vector of the face opposite to vertex $i$.

We recall that ${\0{\mathcal{P}}}_p \Lambda^{1}(t)$ is the N\'{e}d\'{e}lec element of the second kind of degree $p$ with vanishing tangential components on $\partial t$.
\begin{lemma}
We have $\Sigma_{t, p}^{c}={\0{\mathcal{P}}}_p \Lambda^{1}(t),\quad\forall t\in \mathcal{T}$ .
\end{lemma}
\begin{proof}
It is obvious that $\Sigma_{t, p}^{c}\subset {\0{\mathcal{P}}}_p \Lambda^{1}(t)$. To show the converse, we assume $\bm{E}\in {\0{\mathcal{P}}}_p \Lambda^{1}(T)$. Then  from the definition of ${\0{\mathcal{P}}}_p \Lambda^{1}(T)$, $\bm{E}$ vanishes at the vertices  and has the representation:
$$
\bm{E}=\sum_{i=0}^{3}p_{i}\bm{\nu}_{i}, \quad p_{i}\in \mathcal{P}_p (t).
$$
The representation is not unique since there are four normals on a tetrahedron which are not linearly independent. 
We are to prove that $p_{i}$ contains a factor $\lambda_{j}\lambda_{l}\lambda_{m}$, where $i, j, m$ and $l$ are the four different indices chosen from $0, 1, 2, 3$.

On face $f_{j}$, we have
$$
\bm{0}=\bm{E}\times \bm{\nu}_{j}=\sum_{i=0}^{3}p_{i}\bm{\nu}_{i}\times \bm{\nu}_{j}=\sum_{i\neq j} p_{i} \frac{\bm{e}_{lm}}{|\bm{e}_{lm}|},
$$
where $\bm{e}_{lm}$ is the edge connecting vertex $l$ and vertex $m$.

Fixing $i$, we have three options for $j$, so there are three options for $\bm{e}_{lm}$ which are linearly independent and form a basis of $\mathbb{R}^{3}$.  Therefore $p_{i}$ vanishes on $f_{j}$,  $i\neq j$. This implies that $p_{i}$  contains a factor $\lambda_{j}\lambda_{l}\lambda_{m}$.

This proves ${\0{\mathcal{P}}}_p \Lambda^{1}(t)\subset \Sigma_{t, p}^{c}$ and hence $\Sigma_{t, p}^{c}={\0{\mathcal{P}}}_p \Lambda^{1}(t)$.
\end{proof}

We have the following space decomposition which shows that $\mathcal{P}_{2, p}\Lambda^{1}(\mathcal{T}_{h}^{3})$ can be written as the sum of a continuous Hermite element space and local bubble functions: 
\begin{lemma}\label{lem:decomposition-21}
We have $\mathcal{P}_{2, p}\Lambda^{1}(\mathcal{T}_{h}^{3})=\bm{S}_{h}^{p}+\Sigma_{p}^{c}$, where $\bm{S}_{h}^{p}$ is the vector Hermite space, and the restriction of $\Sigma_{p}^{c}$ on an element $t$ coincides with $\Sigma_{t, p}^{c}$.
\end{lemma}
\begin{proof}
First we prove $\bm{S}_{h}^{p}+\Sigma_{p}^{c}\subset \mathcal{P}_{2, p}\Lambda^{1}(\mathcal{T}_{h}^{3})$. In fact, the local polynomials of $\bm{S}_{h}^{p}+\Sigma_{p}^{c}$ and $\mathcal{P}_{2, p}\Lambda^{1}(\mathcal{T}_{h}^{3})$ are the same ($\mathcal{P}_p$).  Furthermore, it is obvious that $\bm{S}_{h}^{p}$ satisfies the interelement continuity imposed by the DoFs of $\mathcal{P}_{2, p}\Lambda^{1}(\mathcal{T}_{h}^{3})$. Now we show that the interelement continuity of the extension by zero of $\Sigma_{t, p}^{c}$ also satisfies the continuity of $\mathcal{P}_{2, p}\Lambda^{1}(\mathcal{T}_{h}^{3})$.

In fact, the ${H}({\curl})$ bubbles  $p_{i}\lambda_{j}\lambda_{l}\lambda_{m}\bm{\nu}_{i}$ vanish on all the edges and the derivatives $\grad(p_{i}\lambda_{j}\lambda_{l}\lambda_{m}\bm{\nu}_{i})$ contain at least two of the barycentric coordinates, which also vanish at all the vertices. This shows that the bubble functions satisfy the interelement continuity.

It remains to show the converse, i.e. $\mathcal{P}_{2, p}\Lambda^{1}(\mathcal{T}_{h}^{3})\subset \bm{S}_{h}^{p}+\Sigma_{p}^{c}$. 
From the DoFs of $\mathcal{P}_{2, p}\Lambda^{1}(\mathcal{T}_{h}^{3})$, one can define the canonical interpolations $I^{c}: \mathcal{P}_{2, p}\Lambda^{1}(\mathcal{T}_{h}^{3})\mapsto \bm{S}_{h}^{p}$.  In fact, given $\bm{u}\in \mathcal{P}_{2, p}\Lambda^{1}(\mathcal{T}_{h}^{3})$,  we can define $I^{c}\bm{u}\in \bm{S}_{h}^{p}$ by defining the function values $I^{c}\bm{u}(\bm{x})$ and derivatives $\partial_{i} I^{c}\bm{u}(\bm{x})$ at the vertices, function values on the edges $I^{c}\bm{u}(e)$ and tangential components on the faces $I^{c}\bm{u}\times \bm{\nu}$ to be the same 
as the corresponding values of $\bm{u}$ ($I^{c}\bm{u}(\bm{x})=\bm{u}(\bm{x})$, $\partial_{i}I^{c}\bm{u}(\bm{x})=\partial_{i}\bm{u}(\bm{x})$, $I^{c}\bm{u}({e})=\bm{u}({e})$,  $I^{c}\bm{u}\times \bm{\nu}=\bm{u}\times \bm{\nu}$).   Then we define the normal components to zero $I^{c}\bm{u}\cdot\bm{\nu}=0$, which is consistent across the boundary of elements. For any $\bm{u}\in \mathcal{P}_{2, p}\Lambda^{1}(\mathcal{T}_{h}^{3})$, it is easy to see that $\bm{u}_{b}|_{t}:=(\bm{u}-I^{c}\bm{u})|_{t}\in {\0{\mathcal{P}}}_p \Lambda^{1}(t)=\Sigma_{t, p}^{c}$. This implies $\bm{u}=I^{c}\bm{u}+\bm{u}_{b}$ can be decomposed as a sum of the Hermite elements and local bubbles.
\end{proof}


\newcase{Element $\mathcal{P}_{2, p}\Lambda^{2}(\mathcal{T}_{h}^{3})\subset H(\div)$.} The construction of $\mathcal{P}_{2, p}\Lambda^{2}(\mathcal{T}_{h}^{3})$ has appeared in Stenberg \cite{Stenberg2010}. The space $\mathcal{P}_{2, p}\Lambda^{2}(\mathcal{T}_{h}^{3})$ can be characterised as:
$$
\mathcal{P}_{2, p}\Lambda^{2}(\mathcal{T}_{h}^{3})=\{ \bm{v}\in H(\div): \bm{v}|_{t}\in (\mathcal{P}_p (t))^{3}, \forall t\in \T_{h}^{3}; \bm{v}\in C^{0}(\V)\}.
$$

The local DoFs are:
\begin{itemize}
\item
 function values of each component $\bm{u}_{i}(\bm{x})$ at each vertex $\bm{x}$, $i=1, 2, 3$,
\item
face DoFs: 
$$
\int_{f}\left (\bm{u}\cdot \bm{\nu}_{f}\right ) \cdot {q},  \quad \forall f\in \mathcal{F}, {q}\in \mathcal{P}_p (f), {q}=0 \mbox{  at the vertices of } f.
$$
\item
interior DoFs
$$
\int_{t} \bm{u}\cdot \bm{v}, \quad \forall \bm{v}\in \mathcal{P}_{p-1}^{-}\Lambda^{1}(t), t\in \mathcal{T}_{h}^{3}.
$$
\end{itemize}

\begin{lemma}
The DoFs for $\mathcal{P}_{2, p}\Lambda^{2}(\mathcal{T}_{h}^{3})$ are locally unisolvent.
\end{lemma}

The global dimension of $\mathcal{P}_{2, p}\Lambda^{2}(\mathcal{T}_{h}^{3})$ is:
$$
\mathrm{dim}(\mathcal{P}_{2, p}\Lambda^{2}(\mathcal{T}_{h}^{3}))=3V+\frac{1}{2}(p^{2}+3p-4)F+\frac{1}{2}(p-1)(p+1)(p+2)T.
$$

\newcase{Element $\mathcal{P}_{2, p}\Lambda^{3}(\mathcal{T}_{h}^{3})\subset L^{2}(\Omega).$}
As above, $\mathcal{P}_{2, p}\Lambda^{3}(\mathcal{T}_{h}^{3})$ is the space of  piecewise polynomials of degree $p$:
$$
\mathcal{P}_{2, p}\Lambda^{3}(\mathcal{T}_{h}^{3}):=\{ q\in L^{2}:  q|_{T}\in \mathcal{P}_p , \forall T\in \mathcal{T} \}.
$$
The  dimension reads
$$
\mathrm{dim}(\mathcal{P}_{2, p}\Lambda^{3}(\mathcal{T}_{h}^{3}))={p+3 \choose 3}T.
$$

We verify the exactness on a contractible domain in the following theorem:
\begin{theorem}
The sequence in 3D ($p\geq 2$)
  \begin{equation}\label{discrete-sequence3d}
\begin{CD}
\mathbb{R}@>>>\mathcal{P}_{2, p+3}\Lambda^{0}(\mathcal{T}_{h}^{3})@>{\grad}>>\mathcal{P}_{2, p+2}\Lambda^{1}(\mathcal{T}_{h}^{3}) @>\curl >>   \mathcal{P}_{2, p+1}\Lambda^{2}(\mathcal{T}_{h}^{3}) @>\mathrm{div} >>\mathcal{P}_{2, p}\Lambda^{3}(\mathcal{T}_{h}^{3}) @ > >>  0
\end{CD}
\end{equation}
is exact on contractible domains.
\end{theorem}

\begin{proof}
From the inf-sup condition of $\mathcal{P}_{2, p+1}\Lambda^{2}(\mathcal{T}_{h}^{3})$ and $\mathcal{P}_{2, p}\Lambda^{3}(\mathcal{T}_{h}^{3})$ which was proved in \cite{Stenberg2010} (it also follows from the inf-sup condition of the Hu-Zhang type vector elements below, where the velocity space is smaller), we see that $\mathrm{div}: \mathcal{P}_{2, p+1}\Lambda^{2}(\mathcal{T}_{h}^{3}) \rightarrow \mathcal{P}_{2, p}\Lambda^{3}(\mathcal{T}_{h}^{3})$ is onto.

We recall that $\mathcal{P}_p \Lambda^{1}(\mathcal{T}_{h}^{3})$ and $\mathcal{P}_p \Lambda^{0}(\mathcal{T}_{h}^{3})$ represent the N\'{e}d\'{e}lec edge element of the second kind of degree $p$ and the Lagrange element of degree $p$ respectively.


From the definition of $\mathcal{P}_{2, p}\Lambda^{1}(\mathcal{T}_{h}^{3})$,  we have
$$
\mathcal{P}_{2, p}\Lambda^{1}(\mathcal{T}_{h}^{3})=\mathcal{P}_p \Lambda^{1}(\mathcal{T}_{h}^{3})\cap \{\bm{w}: \bm{w}\in C^{1}(\V), \bm{w}\in C^{0}(\E)\},
$$
and
$$
\ker \left ({\curl}, \mathcal{P}_{2, p}\Lambda^{1}(\mathcal{T}_{h}^{3})\right )=\ker\left ({\curl}, \mathcal{P}_p \Lambda^{1}(\mathcal{T}_{h}^{3})\right )\cap \{\bm{w}: \bm{w}\in C^{1}(\V), \bm{w}\in C^{0}(\E)\}.
$$
We note that $\grad \mathcal{P}_{p+1}\Lambda^{0}(\mathcal{T}_{h}^{3})=\mathrm{ker}\left (\curl, \mathcal{P}_p \Lambda^{1}(\mathcal{T}_{h}^{3})\right)$ by the exactness of the standard finite element de Rham complex, where
$$
\grad \mathcal{P}_{p+1}\Lambda^{0}(\mathcal{T}_{h}^{3}):=\{\grad u: u\in \mathcal{P}_{p+1}\Lambda^{0}(\mathcal{T}_{h}^{3})\}.
$$

Therefore we have
\begin{align*}
\grad \mathcal{P}_{2, p+1}\Lambda^{0}(\mathcal{T}_{h}^{3})&=\grad \left (\mathcal{P}_{p+1}\Lambda^{0}(\mathcal{T}_{h}^{3})\cap \{s: s\in C^{2}(\V), s\in C^{1}(\E) \}\right )\\&
=\ker\left ({\curl}, \mathcal{P}_{2, p}\Lambda^{1}(\mathcal{T}_{h}^{3})\right ),
\end{align*}
since if $\grad u$ has $C^{1}$  continuity at vertices and $C^{0}$ on edges, $u$ has to be $C^{2}$ and $C^{1}$ at the vertices and on the edges.

It remains to show that for each $\bm{v}_{h}\in \mathcal{P}_{2, p}\Lambda^{2}(\mathcal{T}_{h}^{3})$ satisfying $\div \bm{v}_{h}=0$, we have $\bm{v}_{h}=\curl \bm{w}_{h}$ for some $\bm{w}_{h} \in\mathcal{P}_{2, p+1}\Lambda^{1}(\mathcal{T}_{h}^{3})$. Since we have shown the exactness at other indices, it suffices to check the dimension now.

We summarise the global dimension of the sequence as follows:
  \begin{equation*}
\begin{CD}
1\rightarrow 10V+[2(p-1)+(p-2)]E+{p-1\choose 2}F+{p+2 \choose 3}T \\
\rightarrow 12V+3(p-1)E+2{p+1 \choose 2}F+\left [1/2(p+2)^{3}-(p+2)^{2}-1/2(p+2)+1\right ]T \\
\rightarrow    3V + 1/2(p^{2}+5p)F+1/2p(p+2)(p+3)T 
 \rightarrow  {p+3 \choose 3}  \rightarrow  0.
\end{CD}
\end{equation*}
By straightforward calculations, we know that \eqref{discrete-sequence3d} satisfies the dimension condition of the exactness.
\end{proof}

\subsection{Basis functions in 3D}

For the ${H}(\curl)$ finite element space $\mathcal{P}_{2, p}\Lambda^{1}(\mathcal{T}_{h}^{3})$, we can group the basis functions into several classes. Hereafter, we will use $\psi_{\bm{x}}$ to denote the Hermite nodal basis at a Hermite interpolation point $\bm{x}$, i.e. $\psi_{\bm{x}}(\bm{x})=1$, $\psi_{\bm{x}}(\bm{y})=0$ at any Hermite interpolation point $\bm{y}\neq \bm{x}$, and $\psi_{\bm{x}}$ has vanishing first order derivatives at vertices.
\begin{enumerate}
\item
Vertex-based basis functions: given $\bm{x}\in \mathcal{V}$, its twelve basis functions are 
\begin{align*}
\bm{w}_{\bm{x}, i}&=\psi_{\bm{x}}\bm{e}_{i}, \quad i=1, 2, 3,\\
\tilde{\bm{w}}_{\bm{x}, i, j}&= \tilde{\psi}_{\bm{x}, i}\bm{e}_{j}, \quad i, j=1, 2, 3,
\end{align*}
where $\tilde{\psi}_{\bm{x}, i}$ is the basis function corresponding to the vertex derivative DoF satisfying $\left (\partial^{k}\tilde{\psi}_{\bm{x}, i}\right )\Big |_{\bm{x}}=\delta^{k}_{i}$ and $\tilde{\psi}_{\bm{x}, i}({\bm{y}})=0$ at all the Hermite points $\bm{y}$. Here $\bm{e}_{i}, ~i=1, 2, 3$ are the three bases of $\mathbb{R}^{3}$.

\item
Edge-based basis functions: given a Hermite point $\bm{x}$ on an edge $e$, its associated three basis functions:
$$
\bm{w}_{e, \bm{x}, i}= \psi_{\bm{x}}\bm{e}_{i}, \quad i=1,2,3.
$$

\item
Face-based basis functions: given a Hermite point $\bm{x}$ on a face $f$, its associated two basis functions with tangential directions:
$$
\bm{w}_{f, \bm{x}, i}^{\tau}= \psi_{\bm{x}}\bm{\tau}_{{f}, \bm{x}, i},\quad i=1,2,
$$
where $\bm{\tau}_{f, \bm{x}, i}$ is the tangential vector of the face $f$ at $\bm{x}$.

\item
Face-based basis functions: given a Hermite point $\bm{x}$ on a face $f$, its associated basis functions with the normal direction:
$$
\bm{w}_{f, \bm{x}, i}^{\nu}= \psi_{\bm{x}}|_{t_{i}}\bm{\nu}_{f}, \quad i=1,2,
$$
where $t_{1}$ and $t_{2}$ are the two elements sharing the face $f$, $\bm{\nu}_{f}$ is the normal vector of the face $f$.
\item
Interior basis functions: at each interior Hermite point $\bm{x}$, its three associated basis functions: 
$$
\bm{w}_{t, \bm{x}, i}=\psi_{\bm{x}}\bm{e}_{i}, \quad i=1,2,3.
$$
\end{enumerate}

\subsection{Asymptotic dimensions of the global finite element spaces}

With enhanced smoothness, the dimensions of the global DoFs are significantly reduced. The advantages of the discrete ${H}(\mathrm{div})$ space $\mathcal{P}_{2, p}\Lambda^{2}(\mathcal{T}_{h}^{3})$ have been shown in Stenberg \cite{Stenberg2010}. So here we focus on the ${H({\curl})}$ subspace $\mathcal{P}_{2, p}\Lambda^{1}(\mathcal{T}_{h}^{3})$. An analogous discussion is also possible for the 2D elements.

To see this, we first recall the asymptotic estimates of the dimensions (c.f. \cite{Stenberg2010}):
\begin{align}\label{asymptotic-dimensions}
V=\mathcal{O}\left(\frac{1}{6}T\right),\quad E=\mathcal{O}\left (7V\right )=\mathcal{O}\left (\frac{7}{6}T\right),\quad  F=\mathcal{O}(2T).
\end{align}

\begin{remark}
In 2D, such asymptotic estimates can be established in a rigorous way. From Euler's formula, one has $V-E+F=1$. Since two triangles share one edge and each triangle contains three edges, one further has $2E=3F$ asymptotically. Combining these two identities, one obtains the asymptotic relation $V=\mathcal{O}(1/2F)$. However, in 3D we do not have enough information to give such estimates for general triangulations. To give similar estimates, we consider a special triangulation where each cube is divided into fourteen tetrahedra by connecting the center with eight vertices and the centers of the six faces. In this case, the vertex at the center is connected by  fourteen edges and each edge contains two vertices. Therefore we give an asymptotic estimate $E=7V$. Together with Euler's formula $V-E+F-T=1$, we derived the estimates \eqref{asymptotic-dimensions}.
\end{remark}

We can estimate the dimension of the N\'{e}d\'{e}lec element of the second kind (c.f. \cite{Boffi.D;Brezzi.F;Fortin.M.2013a}):
\begin{align*}
\mathrm{dim}\left (  \mathcal{P}_p \Lambda^{1}\left (\mathcal{T}_{h}^{3}\right )  \right )&=6(p+1)E+4(p+1)(p-1)F+\frac{1}{2}(p+1)(p-1)(p-2)T\\&
=\mathcal{O}\left(\left [ 7(p+1)+8(p+1)(p-1) +\frac{1}{2}(p+1)(p-1)(p-2)   \right ]T\right)\\
&=\mathcal{O}\left(\left (\frac{1}{2}p^{3}+7p^{2}+\frac{13}{2}p \right)T\right).
\end{align*}
For the new element $\mathcal{P}_{2, p}\Lambda^{1}(\mathcal{T}_{h}^{3})$ we have
\begin{align*}
\mathrm{dim}(\mathcal{P}_{2, p}\Lambda^{1}(\mathcal{T}_{h}^{3}))&= 12V+3(p-3)E+2{p-2 \choose 2}F+\left(\frac{1}{2}p^{3}-p^{2}-\frac{1}{2}p+1\right)T\\&
=\mathcal{O}\left(\left [  2+\frac{7}{2}(p-3)+2(p-1)(p-2)+\left(\frac{1}{2}p^{3}-p^{2}-\frac{1}{2}p+1\right)    \right ]T\right)\\
&=\mathcal{O}\left(\left(\frac{1}{2}p^{3}+p^{2}-3p-\frac{11}{2}\right)T\right).
\end{align*}
We see that $\mathcal{P}_{2, p}\Lambda^{1}(\mathcal{T}_{h}^{3})$ has fewer DoFs than $ \mathcal{P}_p \Lambda^{1}\left (\mathcal{T}_{h}^{3}\right ) $:
$$
\mathrm{dim}\left ( \mathcal{P}_p \Lambda^{1}\left (\mathcal{T}_{h}^{3}\right ) \right )-\mathrm{dim}(\mathcal{P}_{2, p}\Lambda^{1}(\mathcal{T}_{h}^{3}))=\mathcal{O}\left( \left (6p^{2}+\frac{19}{2}p+\frac{11}{2}\right )T\right ).
$$
For example, for $p=4$, $\mathrm{dim}\left ( \mathcal{P}_p \Lambda^{1}\left (\mathcal{T}_{h}^{3}\right ) \right )=\mathcal{O}\left(170T\right)$ and $\mathrm{dim}(\mathcal{P}_{2, p}\Lambda^{1}(\mathcal{T}_{h}^{3}))<\mathcal{O}\left(31T\right)$.


\subsection{Hu-Zhang type $H(\mathrm{div})$ space}

The Stenberg $H(\div)$ element $\mathcal{P}_{2, p}\Lambda^{2}$ in 3D does not have continuity on the edges, except for the vertex continuity. On the other hand, the Hu-Zhang element for linear elasticity has continuous normal components on the edges. Using an analogous  idea, we can design another vector $H(\div)$ element with normal continuity on edges. The inf-sup condition of the proposed $H(\div)$-$L^{2}$ finite element pair also holds.

The new element can be described as
$$
\bm{V}_{p}^{h}:=\{\bm{v}\in H(\div; \Omega): v|_{t}\in \mathcal{P}_p (t), \forall t\in \mathcal{T}, \bm{v}\in C^{0}(\mathcal{V}), \bm{v}\cdot\bm{\nu}_{i}\in C^{0}(\mathcal{E}), i=1, 2\}.
$$
The local degrees of freedom are:
\begin{itemize}
\item
function value of each component at each vertex:
 $$
 \quad\bm{u}_{i}(\bm{x}), ~ \bm{x}\in \mathcal{V}, i=1, 2, 3,
 $$
\item
the moments on each edge: 
$$
\int_{e}(\bm{u}\cdot \bm{\nu}_{e, i})w, \quad \forall w\in \mathcal{P}_{p-2}(e), ~e\in \mathcal{E}, i=1, 2,
$$
\item
on each face $f$:
$$
\int_{f}\left (\bm{u}\cdot \bm{\nu}_{f}\right ) {w}, \quad \forall {w}\in  \mathcal{P}_{p-3}(f),
$$
\item
interior DoFs on each tetrahedron $t\in \mathcal{T}$:
$$
\int_{t}\bm{u}\cdot \bm{v},  \quad\forall \bm{v}\in \mathcal{P}^{-}_{p-1}\Lambda^{1}(t).
$$
\end{itemize}

The proof of the unisolvence is similar to Lemma \ref{lem:unisol-2dr1}, following Lemma \ref{interDOF}. From a similar argument as Lemma \ref{lem:decomposition-21},  the interior bubble functions of $\bm{V}_{p}^{h}$ coincide with those of the BDM elements and $\mathcal{P}_{2, p}\Lambda^{2}(\mathcal{T}_{h}^{3})$. 

Following the proof of the inf-sup conditions in Hu and Zhang \cite{hu2015family} (Lemma 3.2), replacing the space of  rigid body motions  by  the space of constants and replacing the symmetric gradient $\epsilon$ by the gradient, we can prove the inf-sup condition of the pair $\bm{V}_{p}^{h}$-$\mathcal{P}_{ p-1}\Lambda^{2}$ where $p\geq 2$.

There is a Lagrange type basis  for $\bm{V}_{p}^{h}$. We recall that $\phi_{\bm{x}}$ is the  Lagrange nodal basis at $\bm{x}$, i.e. $\phi_{\bm{x}}(\bm{x})=1$, $\phi_{\bm{x}}(\bm{y})=0$ for the Lagrange interpolation point $\bm{y}\neq \bm{x}$. For $\bm{V}_{p}^{h}$ we can  group the basis functions into several classes:
\begin{enumerate}
\item
Vertex-based basis functions: given $\bm{x}\in \mathcal{V}$, its three basis functions are 
$$
\bm{v}_{\bm{x}, i}=\phi_{\bm{x}}\bm{e}_{i}, \quad i=1, 2, 3.
$$
\item
Edge-based basis functions: given a Lagrange point $\bm{x}$ on an edge $e$, its associated basis functions with the tangential direction:
$$
\bm{v}_{e, \bm{x}, i}^{\tau}= \phi_{\bm{x}}|_{t_{i}}\bm{\tau}_{e}, 
$$
where $\bm{\tau}_{e}$ is the tangential direction of $e$, $t_{i}$ 
is an element sharing $e$ as an edge.
\item
Edge-based basis functions: given a Lagrange point $\bm{x}$ on an edge $e$, its associated basis functions with normal directions:
$$
\bm{v}_{e, \bm{x}, i}^{\nu}= \phi_{\bm{x}}\bm{\nu}_{e, i},\quad i=1, 2,
$$
where $\bm{\nu}_{e, i}, i=1, 2,$ are the two normal directions of $e$.

\item
Face-based basis functions: given a Lagrange point $\bm{x}$ on a face $f$, its associated two basis functions with tangential directions:
$$
\bm{v}_{f, \bm{x}, i}= \phi_{\bm{x}}|_{t_{i}}\bm{\tau}_{{f}, \bm{x}, i},\quad i=1,2,
$$
where $\bm{\tau}_{f, \bm{x}, i}, i=1, 2,$ are the two tangential vectors of the face $f$ at $\bm{x}$,  $t_{1}$ and $t_{2}$ are the two elements  sharing the face $f$. Functions $\phi_{\bm{x}}|_{t_{i}}$ may take different values when $i=1$ and $i=2$.

\item
Face-based basis functions: given a Lagrange point $\bm{x}$ on a face $f$, its associated basis function with the normal direction:
$$
\bm{v}_{f, \bm{x}}= \phi_{\bm{x}}\bm{\nu}_{f},
$$
where $\bm{\nu}_{f}$ is the normal vector of the face $f$. 

\item
Interior basis functions: at each interior Lagrange point $\bm{x}$, its three associated basis functions: 
$$
\bm{v}_{t, \bm{x}, i}=\phi_{\bm{x}}\bm{e}_{i}, \quad i=1,2,3.
$$
\end{enumerate}

\section{Boundary conditions}\label{sec:bc}

Since the proposed elements have extra smoothness compared with the standard de Rham complexes, the boundary conditions call for more explanations. 

\begin{figure}\vspace{-60pt}
\setlength{\unitlength}{1.2cm}
\begin{center}
\begin{picture}(4.5,4.5)(-2, 0)
\put(-2, 2){\line(1, 0){4}}
\put(-2, 2){\line(1, -1){2}}
\put(0, 0){\line(0, 1){2}}
\put(0, 0){\line(1, 1){2}}
\put(0, 2){\circle*{0.1}}
\put(0, 2){\circle{0.18}}
\put(0, 2.2){$Q$}
\end{picture}
\begin{picture}(4.5,4.5)(-3, 0)
\put(-2, 1){\line(2, 1){2}}
\put(2, 1){\line(-2, 1){2}}
\put(-2, 1){\line(2, -1){2}}
\put(0, 0){\line(0, 1){2}}
\put(0, 0){\line(2, 1){2}}
\put(0, 2){\circle*{0.1}}
\put(0, 2){\circle{0.18}}
\put(0, 2.2){$Q'$}
\end{picture}
\end{center}
\caption{Boundary DoFs. $Q'$ is a corner vertex, while $Q$ is not.}
\label{fig:boundary-DOF}
\end{figure}
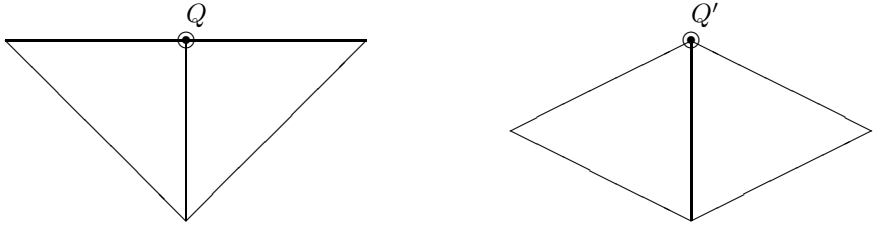

\subsection{Two space dimensions}
We start from two space dimensions and consider discretizations of the following spaces with essential boundary conditions:
$$
H_{0}(\curl; \Omega):=\left \{s\in H(\curl; \Omega), \left . s\right |_{\partial \Omega}=0\right \},
$$
$$
{H}_{0}(\mathrm{div}; \Omega):=\left \{\bm{w}\in {H}(\mathrm{div}; \Omega),  \left .\bm{w}\right |_{\partial \Omega}\cdot\bm{\nu}=0\right \}.
$$
We denote
$$
L^{2}_{0}(\Omega):=\left \{q\in L^{2}(\Omega): \int_{\Omega}q=0\right \}.
$$
Now we are in a position to define the following discrete de Rham sequence with homogeneous boundary conditions: 
  \begin{equation}\label{discrete-sequence01}
\begin{CD}
0@>>> {\0{\mathcal{P}}}_{1, p+2}\Lambda^{0}\left (\mathcal{T}_{h}^{2}\right ) @>\curl>>  {\0{\mathcal{P}}}_{1, p+1}\Lambda^{1}\left (\mathcal{T}_{h}^{2}\right )@>\mathrm{div} >>  {\0{\mathcal{P}}}_{1, p}\Lambda^{2}\left (\mathcal{T}_{h}^{2}\right ) @ > >>  0.
\end{CD}
\end{equation}

To impose the vanishing boundary conditions in \eqref{discrete-sequence01} all  function values of $ \mathcal{P}_{1, p+2}\Lambda^{0}\left (\mathcal{T}_{h}^{2}\right )$ and  normal components of $\mathcal{P}_{1, p+1}\Lambda^{1}\left (\mathcal{T}_{h}^{2}\right )$ should be set  zero on the boundary.  However in the implementation consistent conditions should be considered due to the extra smoothness. For example, $\mathcal{P}_{1, p+2}\Lambda^{0}\left (\mathcal{T}_{h}^{2}\right )$ has two derivative DoFs at each vertex. If boundary values of a function in $\mathcal{P}_{1, p+2}\Lambda^{0}\left (\mathcal{T}_{h}^{2}\right )$ are set  zero, the tangential derivatives along the boundary should also be zero due to the consistency. Therefore in the implementation, the DoFs for the tangential derivatives should be imposed explicitly as well. On the other hand, the normal derivatives should be left free. There is a similar situation for the vertex DoFs of $\mathcal{P}_{1, p+1}\Lambda^{1}\left (\mathcal{T}_{h}^{2}\right )$.

Consistency conditions are related to the geometry of the boundary.  We first follow the terminology in \cite{falk2013stokes} to introduce the definition of corner boundary vertices.
\begin{definition}
A boundary vertex is called a {\it corner vertex} if the two adjacent boundary edges sharing this vertex do not lie on a straight line.
\end{definition}
 In Figure \ref{fig:boundary-DOF}, $Q'$ is a corner vertex while $Q$ is not.

At a non-corner boundary vertex, the two tangential derivatives along its two adjacent edges coincide up to a sign. In this case, we only specify the DoF value of this tangential derivative.
 The number of  global DoFs (unknowns) is reduced by two at this vertex (one function value and one tangential derivative). Otherwise we should specify both derivative DoFs at each corner vertex and the number of DoFs is reduced by three. 

For the $H(\div)$ conforming space $\mathcal{P}_{1, p+1}\Lambda^{1}(\mathcal{T}_{h}^{2})$, we explicitly specify the normal DoFs on each edge and DoFs at each vertex. For a non-corner boundary vertex, we only specify the normal component and for a corner boundary  vertex, we should specify both  components at that vertex.

We denote the number of boundary vertices by $V_{0}$ and the number of non-corner boundary vertices by $V_{0}^{s}$.  Then the following dimension count holds:
$$
{\dim}\left ({\0{\mathcal{P}}}_{1, p}\Lambda^{0}(\mathcal{T}_{h}^{2})\right )={\dim}\left (\mathcal{P}_{1, p}\Lambda^{0}(\mathcal{T}_{h}^{2})\right )- (p-3)E_{0}- 3V_{0}+V_{0}^{s}, 
$$
$$
{\dim}\left ({\0{\mathcal{P}}}_{1, p}\Lambda^{1}(\mathcal{T}_{h}^{2})\right )= {\dim}\left (\mathcal{P}_{1, p}\Lambda^{1}(\mathcal{T}_{h}^{2})\right )-(p-1)E_{0}-2V_{0}+V_{0}^{s},
$$
and
$$
{\dim}\left ({\0{\mathcal{P}}}_{1, p}\Lambda^{2}(\mathcal{T}_{h}^{2})\right )={\dim}\left (\mathcal{P}_{1, p}\Lambda^{2}(\mathcal{T}_{h}^{2})\right )-1.
$$

In fact, for ${\0{\mathcal{P}}}_{1, p}\Lambda^{0}(\mathcal{T}_{h}^{2})$, we remove $p-3$ function value DoFs on each boundary edge, three vertex DoFs at each corner boundary vertex and two DoFs at each non-corner boundary vertex (one function value and one tangential derivative). For ${\0{\mathcal{P}}}_{1, p}\Lambda^{1}(\mathcal{T}_{h}^{2})$, we remove $p-1$ normal DoFs on each boundary edge, two normal DoFs at each corner boundary vertex and one normal DoF at each non-corner boundary vertex.

On contractible domains the number of boundary vertices equals that of boundary edges, i.e. $E_{0}=V_{0}$. Therefore we have
\begin{align*}
&{\dim}\left ({\0{\mathcal{P}}}_{1, p+2}\Lambda^{0}(\mathcal{T}_{h}^{2})\right )+{\dim}\left ({\0{\mathcal{P}}}_{1, p}\Lambda^{2}(\mathcal{T}_{h}^{2})\right )-{\dim}\left ({\0{\mathcal{P}}}_{1, p+1}\Lambda^{1}(\mathcal{T}_{h}^{2})\right )\\
=&{\dim}\left (\mathcal{P}_{1, p+2}\Lambda^{0}(\mathcal{T}_{h}^{2})\right )+{\dim}\left (\mathcal{P}_{1, p}\Lambda^{2}(\mathcal{T}_{h}^{2})\right )-{\dim}\left (\mathcal{P}_{1, p+1}\Lambda^{1}(\mathcal{T}_{h}^{2})\right )+(E_{0}-V_{0})-1\\=&{\dim}\left (\mathcal{P}_{1, p+2}\Lambda^{0}(\mathcal{T}_{h}^{2})\right )+{\dim}\left (\mathcal{P}_{1, p}\Lambda^{2}(\mathcal{T}_{h}^{2})\right )-{\dim}\left (\mathcal{P}_{1, p+1}\Lambda^{1}(\mathcal{T}_{h}^{2})\right )-1=0.
\end{align*}

This shows that the dimension condition of exactness holds for \eqref{discrete-sequence01}. Checking of the exactness at ${\0{\mathcal{P}}}_{1, p+2}\Lambda^{0}(\mathcal{T}_{h}^{2}) $ and ${\0{\mathcal{P}}}_{1, p}\Lambda^{2}(\mathcal{T}_{h}^{2}) $ as the case without boundary conditions, we obtain the exactness of \eqref{discrete-sequence01}.

Boundary conditions for the family $r=2$ are similar. We refer to \cite{falk2013stokes} for detailed discussions.

\subsection{Three space dimensions}

For simplicity of presentation, we assume that the boundary of $\Omega$ is  homeomorphic to the sphere $S^{2}\subset \mathbb{R}^{3}$.
Sobolev spaces with vanishing boundary conditions are defined by
$$
H^{1}_{0}(\Omega):=\left \{v\in H^{1}(\Omega): \left . v\right |_{\partial \Omega}=0 \right\},
$$
$$
H_{0}(\curl; \Omega):=\left \{\bm{u}\in H(\curl; \Omega), \left . \bm{u}\right |_{\partial \Omega}\times \bm{\nu}=0\right \},
$$
$$
{H}_{0}(\mathrm{div}; \Omega):=\left \{\bm{w}\in {H}(\mathrm{div}; \Omega), \left . \bm{w}\right |_{\partial \Omega}\cdot\bm{\nu}=0\right \},
$$
where $\bm{\nu}$ is the unit normal vector of $\partial \Omega$, 
and as a convention, 
$$
L^{2}_{0}(\Omega):=\left \{q\in L^{2}(\Omega): \int_{\Omega}q=0\right \}.
$$
On the continuous level, the complex
$$
\begin{diagram}
0 & \rTo^{} & H^{1}_{0}(\Omega) & \rTo^{\grad} & H_{0}(\curl; \Omega) & \rTo^{\curl} & H_{0}(\div; \Omega) & \rTo^{\div} & L^{2}_{0}(\Omega) & \rTo^{} & 0
\end{diagram}
$$
is exact  on any contractible domain $\Omega$.

As we have seen in the 2D case, whether a boundary vertex is a corner or not depends on the number of independent edges (i.e. edges with linearly independent directions) sharing this vertex. This motivates us to give a similar definition in 3D.
\begin{definition}
A boundary vertex is called a {\it corner vertex} in 3D if the adjacent boundary edges  sharing this vertex are not coplanar.
\end{definition}

Below we take the scalar element space $\mathcal{P}_{2, p}\Lambda^{0}(\Omega)\subset H_{0}^{1}(\Omega)$ as an example to explain how the boundary geometry should be taken into consideration when we impose boundary conditions. 

For a corner boundary vertex in 3D, there are three linearly independent edges (precisely, three edges with linearly independent directions) connected to it. Therefore all derivatives (three first order derivatives and six second order derivatives) at a corner boundary vertex can be derived from given boundary value. On the other hand, at a non-corner boundary vertex, there are only two linearly independent directions along the boundary. Therefore at a non-corner boundary vertex, two tangential first order derivatives along the boundary and three tangential second order derivatives ($\partial_{\tau_{1}}^{2}$, $\partial_{\tau_{2}}^{2}$ and $\partial_{\tau_{1}}\partial_{\tau_{2}}$, where $\tau_{1}$ and $\tau_{2}$ are the two tangent vectors lying on the boundary $\partial \Omega$) can be determined from given boundary data and the corresponding degrees of freedom should be specified to impose boundary conditions. The degrees of freedom corresponding to the normal first order derivative ($\partial_{\nu}$) and the three second order derivatives related to the normal direction, i.e. $\partial_{\nu}^{2}$, $\partial_{\tau_{i}}\partial_{\nu}, ~i=1, 2$, should be treated as unknowns.

For edges in 3D, we similarly define:
\begin{definition}
A boundary edge is called a {\it corner edge} in 3D if the two adjacent faces (2D cells) on the boundary  sharing this edge are not coplanar.
\end{definition}
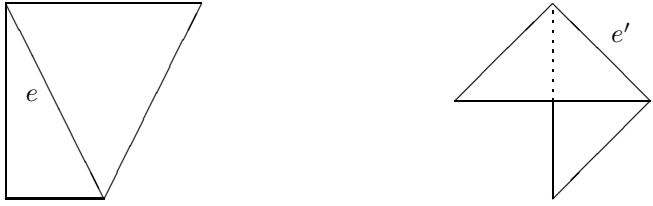
\begin{figure}
\vspace{-60pt}
\setlength{\unitlength}{1.3cm}
\begin{center}
\begin{picture}(4.5,4.5)(-3, 0)
\put(-1, 0){\line(1, 0){1}}
\put(-1, 0){\line(0, 1){2}}
\put(-1, 2){\line(1, -2){1}}
\put(0, 0){\line(1, 2){1}}
\put(-1, 2){\line(1, 0){2}}
\put(-0.8, 1){$e$}
\end{picture}\setlength{\unitlength}{1.3cm}
\begin{picture}(4.5,4.5)(-3, -1)
\put(-1, 0){\line(1, 1){1}}
\put(-1, 0){\line(1, 0){2}}
\put(1, 0){\line(-1, 1){1}}
\put(1, 0){\line(-1, -1){1}}
\put(0, -1){\line(0, 1){1}}
\multiput(0,0)(0, 0.1){10}{\line(0, 1){0.02}}
\put(0.6, 0.6){$e'$}
\end{picture}
\end{center}
\caption{ $e'$ is a corner edge, while $e$ is not.}
\label{fig:boundary-DOF-edge}
\end{figure}
In Figure \ref{fig:boundary-DOF-edge}, $e'$ is a corner edge, while $e$ is not. On a corner edge, derivatives of a function along the two normal directions can be determined by the function value on the boundary. Therefore all DoFs on corner edges should be specified from given boundary data.  On the other hand,  for a non-corner edge, derivative DoFs in the normal direction of the plane cannot be determined from given boundary data.
  In this case, such normal DoFs should be treated as unknowns in the algebraic system.

Boundary conditions for other spaces are similar. The general principle is that we specify all DoFs which can be obtained from given boundary data. 

\subsection{Non-homogeneous boundary conditions}

Paying attention to the geometry of boundaries discussed above, it is trivial to impose vanishing Dirichlet boundary conditions by setting all relevant DoFs zero. However, for non-homogeneous boundary conditions, extra complication may arise. This is usually due to the construction of basis functions.

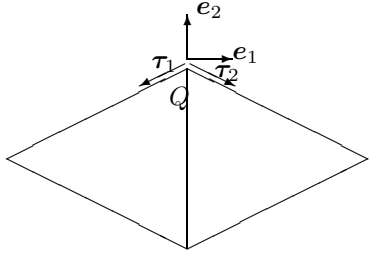
\begin{figure}
\vspace{-40pt}
\setlength{\unitlength}{1.2cm}
\begin{center}
\begin{picture}(4.5,4.5)(-2, 0)
\put(-2, 1){\line(2, 1){2}}
\put(2, 1){\line(-2, 1){2}}
\put(-2, 1){\line(2, -1){2}}
\put(0, 0){\line(0, 1){2}}
\put(0, 0){\line(2, 1){2}}
\put(-0.2, 1.6){$Q$}
\put(-0.03, 2.05){\vector(-2, -1){0.5}}
\put(-0.4, 2){$\bm{\tau}_{1}$}
\put(0.3, 1.9){$\bm{\tau}_{2}$}
\put(0.03, 2.05){\vector(2, -1){0.5}}
\put(0, 2.1){\vector(1, 0){0.5}}
\put(0, 2.1){\vector(0, 1){0.5}}
\put(0.5, 2.1){$\bm{e}_{1}$}
\put(0.1, 2.6){$\bm{e}_{2}$}
\end{picture}
\end{center}
\caption{Tangential/normal DoFs and derivatives along edges.}
\label{fig:tangent-normal-dof}
\end{figure}

For DoFs which should be specified, coefficients in front of the dual basis should be determined by linear combinations of boundary data. For example, in Figure \ref{fig:tangent-normal-dof}, $Q$ is a corner vertex. We use $\bm{e}_{1}$ and $\bm{e}_{2}$ to denote the two canonical basis vectors in $\mathbb{R}^{2}$ and use $\bm{\tau}_{1}$, $\bm{\tau}_{2}$ to denote the vectors corresponding to the two edges sharing $Q$. Assume that $\bm{e}_{1}=a_{1}\bm{\tau}_{1}+a_{2}\bm{\tau}_{1}$ and $\bm{e}_{2}=b_{1}\bm{\tau}_{1}+b_{2}\bm{\tau}_{2}$. Then the directional derivatives satisfy
\begin{align}\label{e-tau-1}
\partial_{\bm{e}_{1}}=a_{1}\partial_{\bm{\tau}_{1}}+a_{2}\partial_{\bm{\tau}_{2}},
\end{align}
and
\begin{align}\label{e-tau-2}
\partial_{\bm{e}_{2}}=b_{1}\partial_{\bm{\tau}_{1}}+b_{2}\partial_{\bm{\tau}_{2}}.
\end{align}
Hence given $\left .u\right |_{\partial \Omega}$, one can obtain $\partial_{\bm{\tau}_{1}}u$ and $\partial_{\bm{\tau}_{2}}u$ by taking derivatives on $\partial \Omega$. Then $\partial_{\bm{e}_{1}}\left ( \left .u\right |_{\partial \Omega}\right )$ and $\partial_{\bm{e}_{2}}\left ( \left .u\right |_{\partial \Omega}\right )$ can be obtained from \eqref{e-tau-1} and \eqref{e-tau-2}. This gives the coefficients in front of the dual basis of $\partial_{\bm{e}_{1}}$ and $\partial_{\bm{e}_{2}}$.

Boundary conditions for other spaces are analogous.


\section{Geometric decomposition}\label{sec:geometric-decomposition}

The geometric locations of the DoFs are essential to define a finite element. Based on this observation, the global finite element space can be decomposed according to the topological entities (c.f. \cite{arnold2009geometric}). Specifically, on a simplex $\Delta$, the dual basis of the DoFs spans a local space with vanishing trace on the boundary of  the patch associated to $\Delta$. This space is independent of the choice of basis. 

This idea is used as the definition of finite elements in Christiansen et al. \cite{christiansen2010finite}. In the finite element systems proposed in \cite{christiansen2010finite}, differential complexes are considered on all the simplexes with different dimensions.  The interelement continuity is guaranteed by requiring that the pull back to each simplex is single-valued, which is enforced by taking an inverse limit. Then the global function spaces can be decomposed as  a sequence of Whiney forms, which has the lowest polynomial degrees but lives on different topological entities, and sequences of bubble functions, each of which lives in the same patch. Although with different names, such a decomposition plays an important role in the study of high order methods, e.g. local complete sequences in Sch{\"o}berl and Zaglmayr \cite{schoberl2005high}, bounded commuting interpolations in Falk and Winther \cite{falk2014local,falk2016bubble}.

In this section, we give a similar decomposition of the new sequences with regularity $r=1$ and $r=2$. Each local space is the span of the dual basis of the DoFs on a topological entity.

Since the continuity involved is usually higher than the natural regularity of the existing de Rham elements in finite element exterior calculus, we will use jets, besides differential forms, in our discussions.
A {\it $r$-jet} means a function together with its Taylor expansion up to order $r$ in a coordinate-free way \cite{saunders1989geometry}. This is a generalisation of the differential forms to higher continuity. Jets and differential operators can be represented in local coordinates. For example, a  $1$-jet for a 2D $1$-form $\bm{v}$ can be represented in the coordinate form $(v^{1}, v^{1}_{1},  v^{1}_{2},  v^{2},  v^{2}_{1},  v^{2}_{2})$, where $v^{1}$ and $v^{2}$ are  the two components of $\bm{v}$ and $v^{i}_{j}$ is the $j$-th derivative of the $i$-th component.  In the definition of jets, $v^{i}$ and $v^{i}_{j}$ are considered to be independent. With such a representation, a differential operator can be identified with its symbol. In the above example, the divergence operator on $1$-form $\bm{v}$ has the coordinate representation $\div \bm{v}=v_{1}^{1}+v_{2}^{2}$. The operators $\grad$ and $\curl$ can be represented in a similar way.

Below we consider the geometric decomposition of the complex
  \begin{equation}\label{vertex-jet-r1}
\begin{CD}
\mathbb{R}@>>>\mathcal{P}_{r, p}\Lambda^{0} (\mathcal{T}_{h}^{n})@>d>> \mathcal{P}_{r, p-1}\Lambda^{1}(\mathcal{T}_{h}^{n}) @>d >>  \cdots @ > d>>  \mathcal{P}_{r, p-n}\Lambda^{n}(\mathcal{T}_{h}^{n}) @ > >>  0.
\end{CD}
\end{equation}
Although our examples  are in $n$D where $n\leq 3$, the local exactness will be verified for any $n\geq 1$ below.

\newcase{Case $r=1$.}


The vertex sequence can be written as
  \begin{equation}\label{vertex-jet-r1}
\begin{CD}
\mathbb{R}@>>>J^{1}\Lambda^{0}(v, \Omega) @>d>> J^{0}\Lambda^{1}(v, \Omega) @>d >> 0  @ > >>  \cdots @ > >>  0.
\end{CD}
\end{equation}
Here $J^{l}(v, \Omega)$ is the $l$-jet at vertex $v$ imbedded in $n$ dimensional space $\Omega$. We note that $\dim (J^{1}\Lambda^{0}(v, \Omega))=n+1$ and $\dim (J^{0}\Lambda^{1}(v, \Omega))=n$ in $n$ dimensions. In local coordinates, $\bm{u}\in J^{1}\Lambda^{0}(v, \Omega) $ and $\bm{w}\in J^{0}\Lambda^{1}(v, \Omega)$ have the form $(u, u_{1}, \cdots, u_{n})$ and $(w^{1}, w^{2}, \cdots, w^{n})$, where $u_{i}=\partial_{i}u$ is considered as an independent variable. 
 By counting the dimensions, \eqref{vertex-jet-r1} is exact.

The edge bubble sequence coincides with the bubble complex of the 1D Hermite-Lagrange pairs:
  \begin{equation}\label{edge-r1}
\begin{CD}
0@>>>{\0{\mathcal{P}}}_{1, p}\Lambda^{0}(e) @>d>>{\0{\mathcal{P}}}_{1, p-1}\Lambda^{1}(e) /\mathbb{R} @ > >>  0.
\end{CD}
\end{equation}
Here $e\in \mathcal{E}$ is a one dimensional simplex.  If $e$ is an interior edge, ${\0{\mathcal{P}}}_{1, p}\Lambda^{0}(e) $ contains functions with vanishing values and derivatives at the vertices and ${\0{\mathcal{P}}}_{1, p-1}\Lambda^{1}(e) /\mathbb{R} $ contains functions with vanishing values at the vertices of $e$ and vanishing integration along $e$. 

The face and interior bubble sequences are the same as the standard finite element de Rham sequences:
  \begin{equation}\label{face-r1}
\begin{CD}
0@>>>{\0{\mathcal{P}}}_p \Lambda^{0}(f) @>d>>{\0{\mathcal{P}}}_{p-1}\Lambda^{1}(f) @>d >> \mathcal{P}_{p-2}\Lambda^{2}(f)  /\mathbb{R}  @ > >>   0,
\end{CD}
\end{equation}
  \begin{equation}\label{volume-r1}
\begin{CD}
0@>>>{\0{\mathcal{P}}}_p \Lambda^{0}(t) @>d>>{\0{\mathcal{P}}}_{p-1}\Lambda^{1}(t) @>d >> {\0{\mathcal{P}}}_{p-2}\Lambda^{2}(t)  @ >d >>   \mathcal{P}_{p-3}\Lambda^{3}(t)  /\mathbb{R}@ > >>0.
\end{CD}
\end{equation}
The vanishing boundary conditions of \eqref{face-r1}-\eqref{volume-r1} are the same as the standard case and  \eqref{edge-r1}-\eqref{volume-r1} are exact.

\newcase{Case $r=2$.}


The vertex sequence can be written as
  \begin{equation}\label{vertex-jet-r2}
\begin{CD}
\mathbb{R}@>>>J^{2}\Lambda^{0}(v, \Omega) @>d>> J^{1}\Lambda^{1}(v, \Omega) @>d>> J^{0}\Lambda^{2}(v, \Omega) @> >> 0.
\end{CD}
\end{equation}
As discussed above,  $J^{l}(v, \Omega)$ is the $l$-jet on vertex $v$ imbedded in an $n$ dimensional domain $\Omega$. We have $\dim (J^{2}\Lambda^{0}(v, \Omega))=(1/2)(n^{2}+3n+2)$, $\dim (J^{1}\Lambda^{1}(v, \Omega))=n(n+1)$ and $\dim(J^{0}\Lambda^{2}(v, \Omega))={n \choose 2}=(1/2)n(n-1)$ in $n$ spatial dimensions. 

In 1D, the complex is exact (no two-forms involved). In 2D, the kernel of $\curl$ (rotation of $\grad$) consists of  constants  and $\div$ ($\rot$): $J^{1}\Lambda^{1}(v, \Omega)\rightarrow J^{0}\Lambda^{2}(v, \Omega)$ is onto, since for $\bm{w}\in  J^{0}\Lambda^{2}(v, \Omega)$ with the coordinate representation $(w)$, we can consider $\bm{v}=(u^{1}, (1/2)w, u^{1}_{2}, u^{2}, u^{2}_{1}, (1/2)w )\in J^{1}\Lambda^{1}(v)$.
From the definition of jets and the differential operators, we see $\div \bm{v}=  \bm{w}$. By analogous argument, we have in 3D: $\ker (\grad)=\mathbb{R}$, and  $\curl$: $J^{1}\Lambda^{1}(v, \Omega)\rightarrow J^{0}\Lambda^{2}(v, \Omega)$ is onto, because for $\bm{w}=(w^{1}, w^{2}, w^{3})\in J^{0}\Lambda^{2}(v, \Omega)$, we can define $\bm{u}=(u^{1}, u_{1}^{1}, 0, 0, u^{2}, w^{3}, u_{2}^{2}, -w^{1}, u^{3}, -w^{2}, 0, u_{3}^{3})$, and $\curl\bm{u}=\bm{w}$.  

 To verify the exactness at $J^{1}\Lambda^{1}(v, \Omega)$, it suffices to check the dimensions. In  $n$ spatial dimensions where $n\geq 1$, we have the dimension count
  $$
0\rightarrow  1\rightarrow \frac{1}{2}\left (  n^{2}+3n+2 \right )\rightarrow n(n+1)\rightarrow  \frac{1}{2}(n-1)n  \rightarrow 0,
 $$
 which verifies the exactness of the vertex sequences \eqref{vertex-jet-r2}. 
 
The edge bubbles of the family $r=2$  reads ($p\geq 2$):
  \begin{equation}\label{edge-r2}
\begin{CD}
0@>>>\0{J}^{1}_{p}\Lambda^{0}(e, \Omega) @>d>>\0{J}^{0}_{p-1}\Lambda^{1}(e, \Omega)  /\mathbb{R}@ > >>  0.
\end{CD}
\end{equation}
Here $\0{J}^{1}_{p}\Lambda^{0}(e, \Omega)$ is the 1-jet on the edge $e$ with vanishing vertex DoFs, i.e. derivatives up to the second order. 
The space $\0{J}^{0}_{p-1}\Lambda^{1}(e, \Omega)/\mathbb{R} $ is the 0-jet on $e$ with vanishing vertex DoFs, i.e. derivatives up to the first order, and vanishing integration.

For $0$- forms, there are $p-5$ DoFs for the function value and $(n-1)(p-4)$ DoFs for the normal  derivatives on each edge. For $1$- forms, we have $n$ components and $p-4$ DoFs for each component on each edge.  Therefore we have the dimension count
 $$
0\rightarrow (p-5)+(n-1)\cdot (p-4)\rightarrow n(p-4)-1\rightarrow 0,
 $$
which implies the exactness.

The face bubbles can be formally written as 
  \begin{equation*}
\begin{CD}
0@>>>{\0{\mathcal{P}}}_{2, p} \Lambda^{0}(f) @>d>>{\0{\mathcal{P}}}_{2, p-1}\Lambda^{1}(f) @>d >>{\0{\mathcal{P}}}_{2, p-2}\Lambda^{2}(f)   /\mathbb{R}@ > >>0.
\end{CD}
\end{equation*}
The vanishing boundary conditions can be  defined as follows. The element ${\0{\mathcal{P}}}_p \Lambda^{0}(f) $ has vanishing  values and derivatives up to the second order at the vertices and vanishing values and first order derivatives   on the edges.  The element ${\0{\mathcal{P}}}_{p-1}\Lambda^{1}(f)  $  has vanishing  values and first order derivatives at the vertices and vanishing  values on the edges. The element ${\0{\mathcal{P}}}_{p-2}\Lambda^{2}(f)  /\mathbb{R}$ takes zero  values at the vertices and has vanishing integral on $f$.

 The interior bubbles coincide with the standard finite element de Rham complexes which are exact:
  \begin{equation*}
\begin{CD}
0@>>>{\0{\mathcal{P}}}_p \Lambda^{0}(t) @>d>>{\0{\mathcal{P}}}_{p-1}\Lambda^{1}(t) @>d >>{\0{\mathcal{P}}}_{p-2}\Lambda^{2}(t)  @ >d >>   \mathcal{P}_{p-3}\Lambda^{3}(t)  /\mathbb{R}@ > >>0.
\end{CD}
\end{equation*}

\section{Bernstein-Gelfand-Gelfand (BGG) constructions}\label{sec:BGG}

Arnold, Falk and Winther \cite{Arnold2006a} introduced the Bernstein-Gelfand-Gelfand (BGG) constructions into numerical analysis to derive finite elements of symmetric tensors from the well-known de Rham elements. Later, Hu and Zhang \cite{hu2014family} designed a new family in a straightforward  way.  In this section, we re-construct the 2D Hu-Zhang stress element using BGG and the families discussed in this paper. We hope that the discussions below could build some connections between these two approaches, and further shed some light on the construction of other tensor valued elements and complexes.  There is a unified   construction in any spatial dimension (c.f. \cite{Hu2015}),  but based on the de Rham families in this paper we only re-construct the 2D case. 

Most of the discussions in this section are routine following \cite{Arnold2006a}. The key observation is that we use the Hermite element for both the 2D ${\0\mathcal{P}}_{1, p}\Lambda^{0}$ space  and each component of the 2D $r=2$ $H(\div)$ space, therefore  $S_{0}$ below is an isomorphism between the finite element spaces.

We construct the 2D Hu-Zhang element step by step based on the following diagram. 
\begin{equation}\label{eqn:diagram-bgg}
 \begin{diagram}
\mathcal{P}_{2, p+3}\Lambda^{0}(\mathcal{T}_{h}^{2}, \mathbb{K}) & \rTo^{d_{0}} & \mathcal{P}_{2, p+2}\Lambda^{1}(\mathcal{T}_{h}^{2}, \mathbb{K}) & \rTo^{d_{1}} & \mathcal{P}_{2, p+1}\Lambda^{2}(\mathcal{T}_{h}^{2}, \mathbb{K})& \rTo & 0\\
  & \ruTo^{S_{0}}   &  &  \ruTo^{S_{1}} &&\\
\mathcal{P}_{1, p+2}\Lambda^{0}(\mathcal{T}_{h}^{2}, \mathbb{V})& \rTo^{d_{0}} & \mathcal{P}_{1, p+1}\Lambda^{1}(\mathcal{T}_{h}^{2}, \mathbb{V})  & \rTo^{d_{1}} &  \mathcal{P}_{1, p}\Lambda^{2}( \mathcal{T}_{h}^{2}, \mathbb{V})  &\rTo & 0
 \end{diagram} 
 \end{equation}
 Here $\mathcal{P}_{r, p}\Lambda^{k}(\mathcal{T}_{h}^{2}, \mathbb{V})$ is the space of differential forms taking values in the 2D vector space $\mathbb{V}=\mathbb{R}^{2}$. Similarly, $\mathcal{P}_{r, p}\Lambda^{k}(\mathcal{T}_{h}^{2}, \mathbb{K})$ is the space of skew-symmetric matrix valued differential forms. In 2D, skew-symmetric matrices can be identified with scalars. We denote the skew-symmetric matrix basis
 $$
 \bm{\chi}:=\left (
 \begin{array}{cc}
 0 & -1\\
 1 & 0
 \end{array}
 \right ).
 $$

The vector valued 0-, 1- and 2-forms in the second row of \eqref{eqn:diagram-bgg} can be represented as
   $$
 \left (
 \begin{array}{c}
 u_{1}\\
 u_{2}
 \end{array}
 \right ), \quad
  \left (
 \begin{array}{c}
 w_{11}\\
 w_{21}
 \end{array}
 \right )dx^{1}+
 \left (
 \begin{array}{c}
 w_{12}\\
 w_{22}
 \end{array}
 \right )dx^{2},\quad 
  \left (
 \begin{array}{c}
 v_{1}\\
 v_{2}
 \end{array}
 \right )dx^{1}\wedge dx^{2},
 $$
 respectively. These vector valued differential forms can be identified with the matrix/vector form:
 $$
\left (
 \begin{array}{c}
 u_{1}\\
 u_{2}
 \end{array}
 \right ),   \quad
 \left (
   \begin{array}{cc}
 -w_{12} & w_{11}\\
 -w_{22} & w_{21}
 \end{array}\right ),\quad 
  \left (
 \begin{array}{c}
 v_{1}\\
 v_{2}
 \end{array}
 \right ).
 $$ 
  The vector valued 1-form $ \mathcal{P}_{1, p+1}\Lambda^{1}(\mathcal{T}_{h}^{2}, \mathbb{V})$ is identified with a matrix and each row is a Stenberg vector element.
  
 Similarly, the skew-symmetric matrix valued differential forms 
 $$
 u\bm{\chi}, \quad w_{1}\bm{\chi}dx^{1}+w_{2}\bm{\chi}dx^{2}, \quad v\bm{\chi} dx^{1}\wedge dx^{2}
 $$
 can be identified with 
  $$
 u\bm{\chi}, \quad \left ( 
 \begin{array}{cc}
 -w_{2}, w_{1}
 \end{array}
 \right )\bm{\chi},
  \quad -v\bm{\chi}.
 $$
 With these identifications, the exterior derivatives $d_{0}$ and $d_{1}$ correspond to $\curl$ and $\div$ for each row.
 
 Following \cite{Arnold2006a,Arnold.D;Falk.R;Winther.R.2006a}, we define $S_{0}: \mathcal{P}_{1, p+2}\Lambda^{0}(\mathcal{T}_{h}^{2}, \mathbb{V})\mapsto \mathcal{P}_{1, p+2}\Lambda^{1}(\mathcal{T}_{h}^{2}, \mathbb{K})$ and $S_{1}: \mathcal{P}_{1, p+1}\Lambda^{1}(\mathcal{T}_{h}^{2}, \mathbb{V})\mapsto \mathcal{P}_{1, p+1}\Lambda^{2}(\mathcal{T}_{h}^{2}, \mathbb{K})$ by
 $$
 S_{0}:  \left (
 \begin{array}{c}
 u_{1}\\
 u_{2}
 \end{array}
 \right )\sim
  \left (
 \begin{array}{c}
 u_{1}\\
 u_{2}
 \end{array}
 \right )\mapsto -u_{2}\bm{\chi}dx^{1}+u_{1}\bm{\chi}dx^{2}\sim - \left (
 \begin{array}{cc}
 u_{1},
 u_{2}
 \end{array}
 \right )\bm{\chi},
 $$
 $$
 S_{1}:  
  \left (
 \begin{array}{c}
 w_{11}\\
 w_{21}
 \end{array}
 \right )dx^{1}+\left (
 \begin{array}{c}
 w_{12}\\
 w_{22}
 \end{array}
 \right )dx^{2}\sim
  \left (
   \begin{array}{cc}
 -w_{12} & w_{11}\\
 -w_{22} & w_{21}
 \end{array}\right )
 \mapsto -(w_{11}+w_{22})\bm{\chi}dx^{1}\wedge dx^{2}\sim (w_{11}+w_{22})\bm{\chi},
 $$
 where the tilde denotes the correspondence between the differential form and the vector form.
 The operator $S_{0}$ is an isomorphism while the operator $S_{1}$ is surjective. With the vector proxy, $S_{0}$ is the identity operator (up to a sign and the skew-symmetric matrix basis $\bm{\chi}$) and $S_{1}$ is the skew-symmetrization. The identity $d_{1}S_{0}+S_{1}d_{0}=0$ holds.

We are ready to introduce the BGG construction now.
\paragraph{Step 1. Complex of product spaces.} Define $\Xi_{p}^{k}=\mathcal{P}_{2, p+1}\Lambda^{k}(\mathcal{T}_{h}^{2}, \mathbb{K})\times  \mathcal{P}_{1, p}\Lambda^{k}( \mathcal{T}_{h}^{2}, \mathbb{V})$. The first step is to establish a new complex of $\Xi_{p}^{k}$ with certain operators.

To do this, we define $\mathcal{A}_{k}: \Xi_{p+1}^{k}\mapsto \Xi_{p}^{k+1}$ by 
$$
\mathcal{A}_{k}:=\left (
\begin{array}{c c}
d_{k} & -S_{k}\\
0 & d_{k}
\end{array}
\right ).
$$
For $(\omega, \mu)\in \Xi_{p}^{k}$, $\mathcal{A}_{k}(\omega, \mu)=(d_{k}\omega-S_{k}\mu, d_{k}\mu)$.
By straightforward calculations, we have $\mathcal{A}_{1}\mathcal{A}_{0}=0$, and the complex
  \begin{equation}\label{BGG-xi}
\begin{CD}
\cdots@>>>\Xi_{p+2}^{0}@>\mathcal{A}_{0}>>\Xi_{p+1}^{1}@>\mathcal{A} _{1}>>\Xi_{p}^{2} @ >>>   0
\end{CD}
\end{equation}
is exact on contractible domains.

\paragraph{Step 2. Projection to a subcomplex.} The second step is to project \eqref{BGG-xi} to a subcomplex. Following \cite{Arnold2006a}, we define 
$$
\Gamma^{0}_{p}=\{(\omega, \mu)\in \Xi_{p}^{0}: d_{0}\omega=S_{0}\mu \}, \quad \Gamma^{1}_{p}=\{(\omega, \mu)\in \Xi_{p}^{1}: \omega=0 \},
$$
and 
$$
\pi^{0}(\omega, \mu)=(\omega, S_{0}^{-1}d_{0}\omega), \quad \pi^{1}(\omega, \mu)=(0, \mu+d_{0}S_{0}^{-1}\omega).
$$
Then it is straightforward to check that the following diagram is commuting, and therefore the bottom complex is exact on contractible domains:
\begin{equation}\label{complex-Gamma}
 \begin{diagram}
 \cdots& \rTo  &  \Xi_{p+2}^{0} & \rTo^{\mathcal{A}_{0}} & \Xi_{p+1}^{1}& \rTo^{\mathcal{A}_{1}} &  \Xi_{p}^{2} & \rTo & 0\\
&&\dTo^{\pi^{0}}  &   & \dTo^{\pi^{1}}    & & \dTo^{\mathrm{id}} &\\
 \cdots& \rTo  & \Gamma^{0}_{p+2}& \rTo^{\mathcal{A}_{0}} & \Gamma^{1}_{p+1}   & \rTo^{\mathcal{A}_{1}} & \Xi_{p}^{2}  &\rTo & 0.
 \end{diagram}
 \end{equation}

\paragraph{Step 3. Identification.}

Identifying $(\omega, S_{0}^{-1}d_{0}\omega)\in \Gamma_{p+2}^{0}$ with $\omega$  and $(0, \mu)\in \Gamma_{p+1}^{1}$ with $\mu$, the bottom sequence in \eqref{complex-Gamma} can be interpreted as
  \begin{equation}
\begin{CD}
\cdots@>>>\mathcal{P}_{2, p+3}\Lambda^{0}\left (\mathcal{T}_{h}^{2}, \mathbb{K}\right )@>d_{0}S_{0}^{-1}d_{0}>>\mathcal{P}_{1, p+1}\Lambda^{1}\left ( \mathcal{T}_{h}^{2}, \mathbb{V}\right )@>(-S_{1}, d_{1}) >>\Xi_{p}^{2} @ >>>   0,
\end{CD}
\end{equation}
where $d_{0}S_{0}^{-1}d_{0}$ corresponds to the Airy operator:
$$
dS_{0}^{-1}d(u\bm{\chi})\sim - \left (
\begin{array}{c c}
 \partial_{2}^{2}u& -\partial_{2}\partial_{1}u \\
-\partial_{1}\partial_{2}u & \partial_{1}^{2}u
\end{array}
\right ).
$$
Now we obtain the elasticity complex with weak symmetry.

\paragraph{Step 4. Imposing the symmetry constraint.}

The matrix form $M$ of $u\in \mathcal{P}_{1, p}\Lambda^{1} \left (\mathcal{T}_{h}^{2}, \mathbb{V}\right )$ is symmetric if and only if $S_{1}u=0$.   We  reduce the DoFs of the skew-symmetric part of $M$ to get the 2D Hu-Zhang elements. 

We first re-write the DoFs for $M$ as:
\begin{enumerate}
\item
four DoFs of function values at each vertex $v\in \mathcal{V}$: $M(v)$,
\item
edge DoFs: $\int_{e}(M\cdot \bm{\nu}_{e})\cdot\bm{q}, \quad \bm{q}\in \left (\mathcal{P}_{p-2}(e)\right )^{2}$,
\item
interior DoFs for the skew-symmetric part, which are the same as the interior DoFs of $\mathcal{P}_{2, p}\Lambda^{2}\left ( \mathcal{T}_{h}^{2}\right )$ ($1/2p^{2}+3/2p-2$ DoFs):
$$
\int_{t}\mathrm{skw}(M): q\bm{\chi}, \quad \forall q\in \mathcal{P}_p (t), ~q=0 \mbox{ on $\mathcal{V}$ },
$$
\item
interior DoFs, which are the same as those of the 2D Hu-Zhang $H_{h}(\div, \mathbb{S})$ space ($3/2p^{2}-3/2p$ DoFs):
$$
\int_{t}M:\theta, \quad \theta \in \{\tau\in \mathcal{P}_p (t, \mathbb{S}): \tau\cdot\bm{\nu}_{\partial t}=0\},
$$
where $\mathcal{P}_p (t, \mathbb{S})$ represents the symmetric matrix valued polynomial space on $t$ with degree $p$. 
\end{enumerate}

The number of interior DoFs is 
$$
\left (1/2p^{2}+3/2p-2\right )+\left ( 3/2p^{2}-3/2p\right )=2\cdot\dim \left ( \mathcal{P}_{p-1}^{-}\Lambda^{1}(t)\right ),
$$
 which coincides with the dimension of the interior bubbles of $\mathcal{P}_{1, p}\Lambda^{1}\left ( \mathcal{T}_{h}^{2}, \mathbb{V}\right )$. Furthermore, if all the DoFs vanish, we first verify that $\mathrm{skw}(M)=0$ by the DoFs of the skew-symmetric part, then $M$ vanishes by the same argument as the Hu-Zhang element \cite{hu2016finite}. This implies the unisolvence, and these DoFs define the same element as $\mathcal{P}_{1, p}\Lambda^{1}\left (\mathcal{T}_{h}^{2}, \mathbb{V}\right )$.

Finally, we can reduce the DoFs of the skew-symmetric part to get the Hu-Zhang element:
\begin{equation}\label{complex-HuZhang}
 \begin{diagram}
 \cdots& \rTo  &  \mathcal{P}_{2, p+3}\Lambda^{0}\left (\mathcal{T}_{h}^{2}, \mathbb{K}\right ) & \rTo^{d_{0}S_{0}^{-1}d_{0}} & \mathcal{P}_{1, p+1}\Lambda^{1}\left (\mathcal{T}_{h}^{2}, \mathbb{K}\right )& \rTo^{(-S_{1}, d_{1})} &  \Xi_{p}^{2} & \rTo & 0\\
&&\dTo^{\mathrm{id}}  &   & \dTo^{\mathrm{id}-i_{h}\mathrm{skw}}    & & \dTo^{\Pi_{h}} &\\
 \cdots& \rTo  &  \mathcal{P}_{2, p+3}\Lambda^{0}\left (\mathcal{T}_{h}^{2}, \mathbb{V}\right )& \rTo^{d_{0}S_{0}^{-1}d_{0}} & \bm{\Sigma}_{h}   & \rTo^{d_{1}} &  \mathcal{P}_{1, p}\Lambda^{2}\left ( \mathcal{T}_{h}^{2}, \mathbb{V}\right ) &\rTo & 0.
 \end{diagram}
 \end{equation}
The operator  $i_{h}$ is a discrete inclusion, mapping a skew-symmetric matrix $K$ to $  \mathcal{P}_{1, p+1}\Lambda^{1}\left (\mathcal{T}_{h}^{2}, \mathbb{V}\right ) $, 
  defined by setting $\mathrm{skw}(i_{h}K)=\mathrm{skw}(K)$ at the vertices and setting the interior DoFs for the skew-symmetric part:
  $$
\int_{t}\mathrm{skw}\left  (i_{h}K\right ): q\bm{\chi}=\int_{t}\mathrm{skw}(K):  q\bm{\chi}, \quad \forall q\in \mathcal{P}_{p+1} (t), q=0 \mbox{ on $\mathcal{V}$ },
$$
and setting the DoFs in Item 2 and Item 4 above to zero.
Defining  $\Pi_{h}(\omega, \mu):=\mu+d_{1}i_{h}\omega$, we can check that \eqref{complex-HuZhang} commutes and particularly, the bottom sequence is exact on contractible domains, which gives the Hu-Zhang stress and displacement elements in 2D.



\section{Concluding remarks}\label{sec:concluding}

We discussed finite element de Rham complexes with higher continuity on  sub-simplexes, and as a result, some of these elements can be represented by Lagrange or Hermite type bases. Motivated by the idea of finite element systems, we gave several different combinations of local exact sequences.  Actually, we found the positions of existing elements with different initial purposes, and discovered new ones (Table \ref{tab:element}): the $H(\curl)$ elements with  $n=3, r=1, 2$ are new.  Table \ref{tab:local-exactness} shows the local  continuity  in 3D: $r=0$ and $r=1$ have the same face bubbles,  $r=0$, $r=1$ and $r=2$ all have the same interior bubbles.

  \begin{table}[H]
\begin{center}
\begin{tabular}{|c|c|c|c|}
\hline
 & $r=0$ & $r=1$ & $r=2$ \\\hline
 $n=1$ &  Lagrange - DG  &   Hermite - Lagrange & Argyris - Hermite   \\\hline
  $n=2$ & Lagrange - BDM - DG & new, 2D Stenberg $H(\div)$ \cite{Stenberg2010} &  Falk-Neilan Stokes \cite{falk2013stokes} \\\hline
   $n=3$ & Lagrange - N\'{e}d\'{e}lec - BDM - DG & new & new, 3D Stenberg $H(\div)$ \cite{Stenberg2010}\\\hline 
\end{tabular}
\end{center}
\caption{Families with $r=0,1, 2$.}
\label{tab:element}
\end{table}%

\begin{table}[H]
\begin{center}
\begin{tabular}{|c|c|c|c|}
\hline
 & r=0 & r=1 & r=2 \\\hline
vertex &  $C^{0}$  &   $C^{1}$ - $C^{0}$ & $C^{2}$ - $C^{1}$- $C^{0}$   \\\hline
edge & $C^{0}$ - $C^{\tau}$  & $C^{0}$ - $C^{\tau}$  &  $C^{1}$ - $C^{0}$  \\\hline
face & $C^{0}$ - $C^{\tau}$ - $C^{n}$  & $C^{0}$ - $C^{\tau}$ - $C^{n}$ & $C^{0}$ - $C^{\tau}$ - $C^{n}$ \\\hline 
\end{tabular}
\end{center}
\caption{Local  continuity of the 3D families: $C^{\tau}$ and $C^{n}$ represent tangential and normal continuity.}
\label{tab:local-exactness}
\end{table}%

The new elements will lead to smaller stiffness and mass matrices due to the enhanced regularity. More importantly, the new  geometric decomposition and topological structure lead to nodal type bases in some cases.

Only complete polynomials were considered. It is known that the $\mathcal{P}^{-}$ family with incomplete polynomials  can also be used to construct exact sequences, i.e. $d: \mathcal{P}_{p}\Lambda^{k}\rightarrow \mathcal{P}_{p-1}\Lambda^{k+1}$, $d: \mathcal{P}_{p}\Lambda^{k}\rightarrow \mathcal{P}_{p}^{-}\Lambda^{k+1}$, $d: \mathcal{P}_{p}^{-}\Lambda^{k}\rightarrow \mathcal{P}_{p-1}\Lambda^{k+1}$ and $d: \mathcal{P}_{p}^{-}\Lambda^{k}\rightarrow \mathcal{P}_{p}^{-}\Lambda^{k+1}$ all have the same kernel (c.f. \cite{Arnold.D;Falk.R;Winther.R.2006a} Lemma 3.8). Combining with the new complexes investigated in this paper, we have more options to construct differential complexes, for example,   
  \begin{equation}\footnotesize
\begin{CD}
0@>>>\mathbb{R}@>>>\mathcal{P}_{1, p}\Lambda^{0}\left (\mathcal{T}_{h}^{3}\right )@>\mathrm{grad}>>\mathcal{P}_{1, p-1}\Lambda^{1}\left (\mathcal{T}_{h}^{3}\right ) @>{\curl} >>\mathcal{P}^{-}_{p-2}\Lambda^{2}\left (\mathcal{T}_{h}^{3}\right )@>{\div} >>\mathcal{P}_{p-3}\Lambda^{3}\left (\mathcal{T}_{h}^{3}\right )   @>{} >>  0,
\end{CD}
\end{equation}  
is also exact on contractible domains.

The dimension count for the local sequences holds for any spatial dimension $n$. Therefore the results in this paper are promising to be generalised to higher spatial dimensions and on other element geometry (e.g. tensor product elements). From the perspective of local sequences, the results presented above could be considered as a nontrivial generalisation of the systematic constructions in Cockburn and Fu \cite{cockburn2017systematic}.

In 3D, the  discrete BGG construction with the new families  in this paper will not yield conforming finite elements for elasticity with strongly imposed symmetry.  Actually the polynomial shape function spaces and the  locality of the degrees of freedom impose a strong constraint in the construction of conforming finite elements. Therefore in the BGG construction, it seems desirable to relax these constraints. Composite elements have been developed in Christiansen and Hu \cite{christiansen2016generalized} which relaxed the constraint of local polynomials to allow piecewise polynomials. Another  approach is to consider nonconforming elements.  A 2D  sequence with the Morley - Crouzeix-Raviart - DG elements can be found in Brenner \cite{brenner2015forty} and similar results also hold for  the Morley-Wang-Xu (MWX) family \cite{ming2006morley}  in higher spatial dimensions.

The new constructions in this paper yield elements with nodal type bases and fewer DoFs. As a result, the structure of the resulting algebraic systems differs from that of the standard vector elements and this provides an opportunity for the construction of well conditioned bases for high order $H(\curl)$ and $H(\div)$ elements. 

On the other hand, although the number of global DoFs is reduced, the nodal bases lead to denser stiffness and mass matrices. Therefore preconditioning and solvers for the new elements remain an issue to be explored. For Hu-Zhang elements for linear elasticity with nodal bases, auxiliary space preconditioners have been designed and analyzed  in  \cite{chen2016fast}. 
Moreover, classical hybridization techniques for canonical  $H(\div)$ face elements (Raviart-Thomas or BDM elements) cannot be directly applied to the nodal  elements (Stenberg or Hu-Zhang type) discussed in this paper.

Furthermore, we hope that the study in this paper could shed some light on the nodal element discretisation for computational electromagnetism.

\section*{acknowledgements}
The authors are grateful to Dr. Rui Ma, Mr. Espen Sande,  Prof. Ragnar Winther and Prof. Jinchao Xu for helpful communications and to the anonymous referees for
valuable suggestions.

\bibliographystyle{spmpsci}      
\bibliography{fes}   

\end{document}